\documentclass{article}
\usepackage[latin1]{inputenc}

\usepackage{cite}
\usepackage{color}

\usepackage{esvect}

\usepackage{graphicx}
\usepackage{caption}
\usepackage{subcaption}
\usepackage[section]{placeins}
\usepackage{listings}
\usepackage{enumerate}

\usepackage{amsmath}
\usepackage{amssymb}

\usepackage{tikz}
\usetikzlibrary{positioning,decorations.markings}

\usepackage{pdfpages}

\newcommand{\R}{{\mathbb R}}

\newcommand{\N}{{\mathbb N}}
\newcommand{\Z}{{\mathbb Z}}

\newcommand{\E}{{\mathrm E\,}}
\newcommand{\Path}{{\mathcal P}}
\newcommand{\B}{{\mathcal B}}

\newcommand{\Alexey}[1]{}
\newcommand{\Fred}[1]{}

\newcommand{\glo}{\alpha}
\newcommand{\maxd}{\beta}
\newcommand{\aon}{\alpha_1}
\newcommand{\atw}{\alpha_2}
\newcommand{\ath}{\alpha_3}
\newcommand{\bad}{\gamma}
\newcommand{\nbad}{\delta}
\newcommand{\roots}{r}
\newcommand{\degree }{D }
\newcommand{\aio}{\chi}

\usepackage{amsthm}
\newtheorem{Claim}{Claim}[section]
\newtheorem{Theorem}[Claim]{Theorem}
\newtheorem{Lemma}[Claim]{Lemma}
\newtheorem{Observation}[Claim]{Observation}
\newtheorem{prop}[Claim]{Proposition}

\newtheorem*{rmk}{Remark}

\usepackage{titlesec}
\usepackage{geometry,blindtext}
\usepackage[font=small,labelfont=bf, width =\textwidth]{caption}

\begin{document}

	\title{Long directed rainbow cycles and rainbow spanning trees}

\author{Frederik Benzing\thanks{Institute of Theoretical Computer Science, ETH, 8092 Zurich, 
		Switzerland. {\tt benzingf@ethz.ch}. }
\and
Alexey Pokrovskiy\thanks{Department of Mathematics, ETH, 8092 Zurich, 
Switzerland. {\tt dr.alexey.pokrovskiy@gmail.com}. 
Research supported in part by SNSF grant 200021-175573.}
\and
Benny Sudakov
\thanks{Department of Mathematics, ETH, 8092 Zurich, Switzerland. 
{\tt benjamin.sudakov@math.ethz.ch}. 
Research supported in part by SNSF grant 200021-175573.}
}
\maketitle
	\abstract{	A subgraph of an edge-coloured graph is called rainbow if all its
edges have different colours.
The problem of finding rainbow subgraphs goes back to the work of
Euler on transversals in Latin squares and was extensively studied
since then.
In this paper we consider two related questions concerning rainbow
subgraphs of complete, edge-coloured graphs and digraphs.
In the first part, we show that every properly edge-coloured  complete
directed graph contains a directed rainbow cycle of length
$n-O(n^{4/5})$. This is motivated by an old problem of Hahn and improves a result of
Gyarfas and Sarkozy.
In the second part, we show that any tree $T$ on $n$ vertices with
maximum degree $\Delta_T\leq \maxd n/\log n$
has a rainbow embedding into a properly edge-coloured  $K_n$ provided
that every colour appears at most  $\glo n$ times
and  $\glo, \maxd$ are sufficiently small constants.}
	
	\section{Introduction}
	In this paper we study rainbow subgraphs of properly edge-coloured complete graphs and digraphs. An edge-colouring of an undirected graph is \textit{proper} if no two edges sharing a vertex have the same colour. In the directed setting, no pair of edges with a common start point and  no pair of edges with a common end point may be monochromatic. In both cases a subgraph of the complete graph is \textit{rainbow} if all its edges have distinct colours.
	We define the complete directed graph on $n$ vertices, denoted $\overleftrightarrow{K_n}$, to be the graph on $n$ vertices with an edge going in both directions between any two distinct vertices.

	The search for rainbow structures can be traced back to the 18th century, when Euler initiated the study of transversals in Latin squares. In the meantime, a multitude of conjectures and results in this field has been published, see e.g. \cite[Chapter 9]{wanless} for a survey. Let us first define the notions involved and give their natural reformulation in terms of rainbow subgraphs.	A \textit{Latin square of order $n$} is an $n\times n$ array filled with symbols, so that each symbol appears precisely once in each row and column. A \textit{partial transversal of length $k$} is a collection of $k$ cells of the Latin square, so that no two cells share their row, column or symbol. 
	To every $n\times n$ Latin square, one can assign a proper colouring of $\overleftrightarrow{K_n}$ with a loop added at each vertex as follows---label the vertices of $\overleftrightarrow{K_n}$ by the numbers $1, \dots, n$ and colour each directed edge $ij$ by the symbol in the cell $(i,j)$. 
	Identifying the cell $(i,j)$ with the edge $ij$, a partial transversal now corresponds to a rainbow subgraph of $\overleftrightarrow{K_n}$ in which each vertex has in- and out-degree at most 1. 
	
	A long-standing conjecture attributed (in slightly different versions) to Ryser \cite{Ryser}, Brualdi \cite{Brualdi}, and Stein \cite{Stein} asks whether every Latin square of order $n$ contains a partial transversal of length $n-1$. The best known approximate version of this, due to Hatami and Shor \cite{Shor},  asserts that there always is a partial transversal of length $n-O\left(\log^2 n\right)$, improving several earlier results.	
	 A partial transversal is called \textit{cycle-free} if the corresponding subgraph of $\overleftrightarrow{K_n}$ is cycle-free. Since partial transversals correspond to rainbow subgraphs of $\overleftrightarrow{K_n}$ with maximum in- and out-degree 1, cycle-free partial transversals correspond to rainbow subgraphs of $\overleftrightarrow{K_n}$ which are unions of vertex-disjoint directed paths.
	  We shall call such subgraphs \textit{path forests}. Gy\'arf\'as and S\'ark\"ozy \cite{gyarsar} conjectured that every Latin square contains a partial, cycle-free transversal of length $n-2$ and showed that this would be best possible for $n=4$. They proved that every Latin square contains a partial, cycle-free transversal of size $n-O(n\log\log n/\log n)$. One of our main results improves upon and generalises this result.
	\begin{Theorem}\label{TLRC}
		Every properly edge-coloured $\overleftrightarrow{K_n}$ contains a rainbow path forest of length $n-O(n^{2/3})$ and a rainbow cycle of length $n-O(n^{4/5})$. 
	\end{Theorem}
 Note that this result is more general than that of Gy\'arf\'as and S\'ark\"ozy \cite{gyarsar}. On the one hand it shows the existence of long cycles rather than just path forests, on the other hand it only assumes $\overleftrightarrow{K_n}$ to be properly coloured, rather than each colour-class being a 1-factor (which is the case for colourings of $\overleftrightarrow{K_n}$ coming from Latin squares). 
	
	Theorem \ref{TLRC} is also interesting since it generalises a recent result by Alon and the second two authors \cite{longcycle}. They proved that every properly edge-coloured $K_n$ contains a rainbow cycle of length $n-O(n^{3/4})$ and our proof of Theorem \ref{TLRC} is based on their key ideas. 
	Both these results raise the question of how far the error terms, $O(n^{4/5})$ and $O(n^{3/4})$ respectively, can be pushed.  
	In the undirected case, Balogh and Molla~\cite{BaloghMolla} reduced the  $O(n^{3/4})$ error term to $O(n^{1/2}\log n)$.
	Hahn \cite{Hahn} conjectured that every properly coloured complete graph contains a rainbow Hamilton path. However, Maamoun and Meyniel \cite{m&m} constructed an edge-colouring of $K_n$, for $n=2^k$, which does not contain a rainbow Hamilton path, disproving Hahn's conjecture. Their construction generalises easily to the directed case: Given a colouring $K_n$ which does not contain a Hamilton path, replace each undirected edge $\{u,v\}$ of colour $f$ by two directed edges $(u,v),(v,u)$ giving both of them colour $f$. Then this colouring of $\overleftrightarrow{K_n}$ does not contain a directed (or, in fact, undirected) Hamilton path.
		
	Beyond Latin squares, rainbow structures play an important role in Ramsey Theory, for they are one of the structures guaranteed by the canonical theorem of Ramsey Theory, proved by Erd\H{o}s and Rado \cite{canRam}. In this setting it is a natural question under which conditions a given graph $G$ has a rainbow embedding into an edge-coloured $K_n$. If such an embedding exists, we call the colouring of $K_n$ $G$-\textit{rainbow}. We also say that an edge-colouring is \textit{locally $k$-bounded} if every colour appears at most $k$ times at each vertex. Analogously, the colouring is \textit{globally $k$-bounded} if every colour appears at most $k$ times in total.
	
	In 1976 Bollob\'as and Erd\H{o}s \cite{boerd} raised the question which edge-colourings contain a properly coloured Hamilton cycle and proved that locally $\alpha n$-bounded colourings do, for $\alpha =1/69$. They conjectured that their statement holds for any $\alpha <1/2$. After several improvements, Lo \cite{LoHam} proved an asymptotic version of this conjecture. 
	
	Since then different generalisations have been made, asking for rainbow rather than properly coloured graphs. E.g.,  Hahn and Thomassen \cite{HaThrain}  formulated an analogue to the Bollob\'as-Erd\H{o}s conjecture in 1986, suggesting that there exists an $\alpha$ so that any globally $\alpha n$-bounded colouring contains a rainbow Hamilton cycle. They proved that any globally $k$-bounded colouring contains a rainbow Hamilton cycle for $k=\alpha n^{1/3}$. Other sublinear values for $k$ were proved by by Erd\H{o}s, Ne\v{s}et\v{r}il, and R\"odl \cite{neset}, and by Frieze and Reed \cite{FrieR}. Albert, Frieze and Reed \cite{AlFrieReed} finally showed that any globally $n/64$-bounded colouring contains a rainbow Hamilton cycle. 
	
	Frieze and Krivelevich posed the question whether there is an $\alpha=\alpha(\Delta)$ so that any globally $\alpha n$-bounded colouring contains a copy of every spanning tree with maximum degree $\Delta$. B\"ottcher, Kohayakawa and Procacci \cite{globoundedrain} showed a stronger result, namely that any globally $\frac{n}{51\Delta^2}$-bounded colouring contains a rainbow copy of any graph (and not only tree) with maximum degree at most $\Delta$. The third author and Volec \cite{VolSud} showed that the $\Delta$-dependence in this theorem is best possible. More precisely, they constructed a locally 3-bounded and globally 9-bounded colouring of $K_{3n}$ which is not $T$-rainbow for any tree $T$ of radius 2. This especially holds for radius-2-trees with maximum degree $O\left(\sqrt{n}\right)$ showing that the $\Delta$-dependence is indeed optimal.  See \cite{CoulsonPerarnau, VolSudKam} for additional work in this area when the host graph $K_n$ is replaced by a different graph.
	
	In light of the construction from~\cite{VolSud}, we consider the question whether properly coloured (i.e. locally 1-bounded) colourings behave differently than locally $3$-bounded colourings. More precisely, we ask under which conditions a proper colouring of $K_n$ is $T$-rainbow for a given spanning tree $T$. The intuition that we might be able to find a rainbow embedding of any spanning tree into any properly coloured $K_n$ is false. Recall the colouring from Maamoun and Meyniel \cite{m&m} of $K_n$, for $n=2^k$, which does not contain a Hamilton path. We also observe that the colouring they present does not allow for a rainbow embedding of any spanning tree in which all but precisely 2 vertices have odd degree, giving a much wider class of counterexamples. The details of this and other colourings of $K_n$ which do not contain rainbow copies of certain spanning trees are presented in Section \ref{conc}.
	
	The approach of B\"ottcher, Kohayakawa and Procacci \cite{globoundedrain}, which uses a framework for the Local Lemma developed by Lu and Sz\'ekely \cite{LS}, can be modified straightforwardly to show that any globally $\frac{n}{C\Delta}$-bounded (where $C$ is a sufficiently large constant), proper colouring of $K_n$ contains a rainbow copy of every tree with maximum degree $\Delta$. Combining this method with some additional ideas, we show that the condition of being properly coloured gives much stronger results, distinguishing them from locally 3-bounded colourings, for example.
	\begin{Theorem}\label{TRT}
		There are constants $\glo,\maxd>0$ so that the following holds for every integer $n$.
		Let $T$ be a tree on $n$ vertices with maximum degree at most $\maxd n/\log n$ and let $c$ be a proper colouring of $K_n$ which is globally $\glo n$-bounded. Then $c$ is $T$-rainbow.	
	\end{Theorem}
	Recall that the condition of the colouring being globally $\glo n$-bounded cannot simply be dropped. However, we believe that the statement might be true for any $\glo < 1/2$. It is also possible, that the condition on the maximum degree of the tree $T$ can be dropped. 
		
	The proofs of Theorem \ref{TLRC} and \ref{TRT} are given in Section \ref{LRC} and \ref{RT} respectively. Each section starts with a short outline of the main steps of the proof. In Section \ref{conc}, we end with some remarks and open questions.

\subsection*{Notation}
	For any graph $G$, let $V(G)$ be the set of vertices of $G$, let $E(G)$ be the set of edges of $G$ and write $v(G)=|V(G)|$ as well as $e(G)=|E(G)|$.
	
	For two vertices $v,v'$ let $vv'$ denote the directed edge from $v$ to $v'$. Suppose we are given a colouring $c$ of the edges of $G$. For an edge $e$ let $c(e)$ be its colour. For a graph $G$, let $C(G)$ denote all the colours appearing in $G$.
	
	For a directed path $P=v_1\rightarrow\ldots\rightarrow v_{|P|}$ denote by $f(P):=v_1$ the first vertex of $P$ and by $l(P):=v_{|P|}$ the last vertex of $P$. 
	If $v\in P$ with $v\neq l(P)$, then let $v^+$ be the successor of $v$ on $P$ and write $c(v):=c(vv^+)$. If $v\in P$ with $v\neq f(P)$, then let $v^-$ be the predecessor of $v$ on $P$. For subsets of vertices $X\subset V(P)\setminus l(P)$, define $X^+=\{x^+:x\in X\}$. Define $X^-$ similarly.
	
	We call a pairwise vertex-disjoint collection $\Path=\{P_1,\ldots, P_r\}$ of directed paths  a \textit{path forest}. Given a path forest $\Path$ as above, define $f(\Path)=\{f(P_i):\, 1\leq i\leq r\}$ and $l(\Path)=\{l(P_i):\, 1\leq i\leq r\}$. 
	If $v\in V(\Path)\setminus l(\Path)$, then let $v^+$ be the successor of $v$ in $\Path$ and also write $c(v)=c(vv^+)$. Also define $v^-$ and $X^+, X^-$ for suitable $v\in V(\Path), X\subset V(\Path)$ in the expected way.
	
	For two paths $P,Q$, which only intersect in the vertex $l(P)=f(Q)$, we write $P+Q$ for the concatenation of $P$ and $Q$.
	
	For sets of vertices $A,B\subset V(G)$ and a subset of colours $C\subset C(G)$, let $d(A,B,C)$ denote the number of edges from $A$ to $B$ whose colour lies in $C$. Moreover, we write $d^-(A,C):=d(V(G),A,C)$ and similarly $d^+(A,C)=d(A,V(G),C)$. If we drop $C$ in this notation, we assume $C=C(G)$, e.g. $d(A,B):=d(A,B,C(G))$. Also, if $A=\{v\}$ consists of a single vertex, we will not resist the temptation of writing $d(v,B)$ rather than $d(\{v\},B)$. For a vertex $v\in V(G)$ we also write $N^-(v)=\{w\in V(G):\;wv\in G \}$.

	\section{Long Directed Rainbow Cycles}\label{LRC}
	In this section we prove Theorem \ref{TLRC}. Before giving the details we outline the strategy: We will first put aside a few colours from $\overleftrightarrow{K_n}$ to from an expander $H$. The remaining graph $\overleftrightarrow{K_n}\setminus H$ will have high minimum degree and colours disjoint from $H$. In this graph we will find a long rainbow path forest (a vertex-disjoint collection of paths). In the final step we will use edges from $H$ to `rotate' and `glue together' the rainbow path forest to obtain a long cycle. This technique was also used in \cite{longcycle} for the undirected case and we show here how to transfer this technique to the directed case.

	\subsection{Random Subgraph}
	This section is concerned with `random subgraphs' of a properly coloured $\overleftrightarrow{K_n}$. Here by `random subgraph' we mean the subgraph that for each colour of $\overleftrightarrow{K_n}$ contains all edges of that colour with probability $p$ and otherwise contains no edge of that colour, where the coin is thrown independently for each colour. The theorem below assures that such a subgraph looks almost like a truly random graph. The result was originally proved by Alon and the last two authors  \cite{longcycle} for properly edge-coloured, complete, undirected graphs. It is also pointed out that the argument can be generalised to different settings. The version we need, stated below, follows from Theorem 5.1 in \cite{longcycle} and the remarks thereafter.
	\begin{Theorem}[Alon, Pokrovskiy, and Sudakov \cite{longcycle}]\label{rand}
		Given a proper edge-colouring of $\overleftrightarrow{K_n}$, let $G$ be the random subgraph obtained by choosing each colour class indpendently with probability $p$ satisfying $\log(n)/n \ll p\leq 1/2$. Then, with high probability, all vertices in $G$ have in-degree and out-degree $(1-o(1))np$ and for every two disjoint subsets $A,B$ with $|A|,|B|\gg (\log n/p)^2,\,e_G(A,B)\geq (1-o(1))p|A||B|.$
	\end{Theorem}

	\subsection{Long Rainbow Path Forest}
	The following Lemma shows that we can find long rainbow path forests in properly coloured digraphs with high in-degree. It is based on  a technique by Andersen \cite{Andersen}, developed for undirected graphs. Its adaption to the directed case requires an additional idea relying on the following simple observation. 
	\begin{Observation}\label{add}
		Let $\Path$ be a path forest and let $v\in l(\Path)$ and $f,f'\in f(\Path)$ be distinct. Then, at least one of $\Path+vf$ and $\Path+vf'$ is still a path forest. 
	\end{Observation}
	\begin{proof}
		Observe that a path forest is precisely a graph which contains no cycle and has maximum in-degree and maximum out-degree at most 1. In $\Path+vf$ every vertex still has in-degree/out-degree at most 1. Thus, if $\Path+vf$ is not a path forest, then it must contain a cycle. This implies that $v$ and $f$ lie on the same path of $\Path$. But then $v$ and $f'$ lie on different paths in $\Path$, and hence $\Path+vf'$ is a path forest.
	\end{proof}
	We record another simple observation that will help us later.
	\begin{Observation}\label{pathsize}
		Let $\Path$ be a path forest, $v\in l(\Path)$ and $f\in f(\Path)$ as before and assume that $\Path'=\Path+vf$ is a path forest. We then have $|\Path'|=|\Path|-1$, meaning that $\Path'$ consists of one path less than $\Path$. Also, deleting an edge (but not the vertices it is incident to) from any path forest increases its size (number of paths) by 1.
	\end{Observation}
	
	We now state and prove the main Lemma of this section, guaranteeing long rainbow path forests in properly coloured digraphs with high in-degree.
	
	\begin{Lemma}\label{long}
		Let $n\in\N$ and $0<\gamma\leq\delta<1$ satisfy 
		$$
		\left\lfloor\frac{\gamma n}{\lfloor 1/\delta\rfloor}\right\rfloor \frac{1}{\lfloor 1/\delta\rfloor} > 1.
		$$
		Then any properly coloured digraph $G$ on $n$ vertices with minimum in-degree at least $(1-\delta)n$ contains a directed rainbow path forest $\mathcal{P}=\{P_1,\ldots,P_r\}$ with $r\leq \gamma n$ and $\vert e(\mathcal{P})\vert \geq (1-3\delta)n$. 
	\end{Lemma}
	\begin{rmk} In our applications of Lemma \ref{long} we will have $\delta\rightarrow0$ and $\delta \gamma n \rightarrow \infty$, so that  
		$$
		\left\lfloor\frac{\gamma n}{\lfloor 1/\delta\rfloor}\right\rfloor \frac{1}{\lfloor 1/\delta\rfloor} = \left(1+o(1)\right) \delta^2 \gamma n.
		$$
	\end{rmk}
	
	\begin{proof}
		For the sake of readability, we shall assume that $1/\delta, \gamma n, \delta \gamma n$ are all integers, so that the condition on $\gamma, \delta, n$ simplifies to 
		\begin{equation}\label{readable}
		\delta^2 \gamma n > 1.
		\end{equation}
		
		Let $\mathcal{P}=\{P_1,\ldots,P_r\}$ be a maximum (w.r.t. $e(\mathcal{P})$) directed rainbow path forest with $r\leq \gamma n$. Assume for contradiction that $|e(\mathcal{P})|<(1-3\delta)n$. We will show that $\Path$ cannot be maximum. Note that we may assume that $r=\gamma n$, since if $r<\gamma n$, then we can simply add single vertices to $\Path$ in order to increase the number of paths. This is possible because $|V(G)\setminus V(\mathcal{P})|=n - |V(P)|\geq n- |e(P)|-r\geq 3\delta n- \gamma n\geq \gamma n$. 
		
		We partition the set $f(\Path)$ into sets $Q_0,\ldots,Q_s$ of equal size  $1/\delta$, where $s:=\delta\gamma n-1$, in other words: $$f(\Path)=\bigcup_{i=0}^s Q_i, \text{ where the union is disjoint and }|Q_i|=1/\delta \text{ for } i=0,\ldots,s.$$
		
		Define $C_0 = C(G)\setminus C(\Path)$ to be the set of colours not appearing in $\mathcal{P}$. For $1\leq i\leq s+1$ we define subsets $V_i\subset V(G)$ and $C_i\subset C(G)$ recursively as follows:
		\begin{equation*}
		\begin{aligned}
		V_i &:= \{v\in V(G) : \, d(v,Q_{i-1},C_{i-1})\geq 2 \} \\
		C_i &:= C_{i-1} \cup \{c(v): \, v\in V_i\cap V(\Path)\setminus l(\Path) \}.
		\end{aligned}
		\end{equation*}
		\begin{Observation}\label{newcolour}
			Note, that $c\in C_i\setminus C_{i-1}$ for some $i>0$ implies that there is some $v\in V_i\cap V(\Path)\setminus l(\Path)$ with $c(v)=c$ and such that there are at least two edges from $v$ to $Q_{i-1}$ whose colours lie in $C_{i-1}$. 
		\end{Observation}
		
		Our plan is to show that $|C_i|\geq |C_{i-1}|+ \delta n$. Once we establish this, we are almost done: We then have $|C_{s+1}|-|C_0| \geq (s+1) \delta n= \delta^2\gamma n^2  > n$ by \eqref{readable}. On the other hand $C_{s+1}\setminus C_0\subset C(\Path)$ by construction, so that $|C_{s+1}|-|C_0|\leq n$. This contradiction finishes the proof.
		
		It remains to prove $|C_i|\geq |C_{i-1}|+ \delta n$. In order to do so, we establish the following claim which relies on the maximality of $\Path$.
		\begin{Claim}\label{algoclaim}
			We have $V_i\subset V(\Path)\setminus l(\Path)$ for $i=1,\ldots,s+1$.	
		\end{Claim}		
	\begin{proof}
			We prove the claim by contradiction, i.e. we assume that there exists an index $i$ such that $V_i\not\subset V(\Path)\setminus l(\Path)$ and deduce that $\Path$ is not a maximum rainbow path forest. Let $i_1$ be the smallest index with $V_{i_1}\not\subset V(\Path)\setminus l(\Path)$ and choose some $v_1\in V_{i_1}\cap (l(\Path)\cup (V(G)\setminus V(\Path)))$. For now, assume $v_1\in l(\Path)$. The case $v_1\in V(G)\setminus V(\Path)$ will be very similar and treated at the end of the proof. 
			
			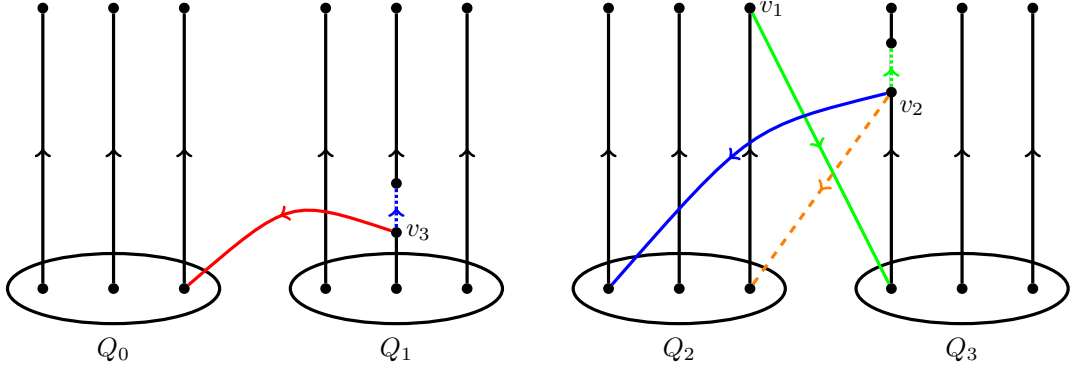
\begin{figure}[t]
				\begin{tikzpicture}[scale=0.93, inner sep=1,
				vertex/.style={circle,draw,fill,}
				]
				
				\begin{scope}[very thick,decoration={
					markings,
					mark=at position 0.5 with {\arrow{>}}}
				] 
				\clip (-2,-1.1) rectangle (14,4.1); 
				\foreach \x in {0,1,...,3}{
					\draw (4*\x,0) ellipse (1.5 and 0.5) node[below=0.6] {$Q_{\x}$};
					\foreach \y in {-1,0,1}{
						\node (start\x_\y) at (4*\x+\y,0) [vertex]{};
						\node (end\x_\y) at (4*\x+\y,4) [vertex]{};
						\draw[postaction={decorate},very thick] (start\x_\y)--(end\x_\y);
					}
				}
				\draw[postaction={decorate}] (end2_1) [green]-- (start3_-1);
				\node[vertex,label=right:$v_1$] at (end2_1) {}; 
				
				\node  (v2) at (11,2.8) [vertex,label=below right:$v_2$]{}; 
				\node  (v2p) at (11,3.5) [vertex]{};
				\draw (v2) [white]--(v2p); 
				\draw[postaction={decorate}] (v2) [green, densely dotted]--(v2p);
				
				\draw[postaction={decorate}] (v2) [orange,dashed]-- (start2_1);
				
				\draw[postaction={decorate},blue] (v2) [blue] ..controls(9,2.3).. (start2_-1);
				
				\node[vertex, label=right:$v_3$] (v3) at (4,0.8){}; 
				\node[vertex] (v3p) at (4,1.5){};
				\draw (v3) [white]-- (v3p);
				\draw[postaction={decorate}] (v3) [blue, densely dotted]-- (v3p);
				
				\draw[postaction={decorate},red] (v3) .. controls (2.5,1.3) .. (start0_1); 
				
				\end{scope}
				\end{tikzpicture}
				
				\caption{An example of how the algorithm in the proof of Claim \ref{algoclaim} might work. Suppose $i_1 = 4$, i.e. there is a vertex $v_1 \in V_4$, so that there are two edges from $v_1$ to $Q_3$ using colours in $C_3$. In the first step, the algorithm adds one of these edges (the solid green edge in our case) to the rainbow path forest and then deletes an edge of the same colour (the dotted green edge). Now, $i_2$ is the smallest index, so that `green' is an element of $C_{i_2}$, suppose $i_2 = 3$. Moreover, $v_2$ is the starting point of the deleted (green dotted) edge. There are two edges from $v_2$ to $Q_2$ (drawn as orange dashed and blue solid lines). In the second step, the algorithm will add one of them to the path forest. Adding the orange edge would create a cycle. Thus, the algorithm uses the other possibility and adds the blue, solid edge. It then deletes the dashed blue edge, to keep the path forest rainbow. Suppose we have $i_3 = 1$. In the last step, the algorithm adds the red edge and terminates, since there was no other red edge in the path forest. The resulting rainbow path forest contains one more edge than the original path forest contradicting maximality of $\mathcal{P}$.}\label{algo}
			\end{figure}
			
			We will apply an algorithm that transforms the original rainbow path forest $\Path_1=\Path$ into a rainbow path forest $\Path_{T+1}$ with one more edge than $\Path$ in $T$ steps. The notation needed for the formal description of this algorithm makes the proof a bit technical, we provide an instance of the algorithm in Figure \ref{algo}. 
			
			The algorithm will consist of $T$ steps, where the $k$-th step of the algorithm will have as input an index $i_k$, a vertex $v_k$, and a rainbow path forest $\Path_{k}$ and it will generate an output $i_{k+1},v_{k+1},\Path_{k+1}$. The path forest $\Path_{k+1}$ will be obtained from $\Path_{k}$ by adding one edge to $\Path_{k}$ and deleting one from it. The only exception from this is the last step, i.e the $T$-th step, which will only output $\Path_{T+1}$, which will be obtained by adding one edge to $\Path_T$ (without deleting one). 
			
			All in all, the algorithm will generate a sequence of indices $i_1,\ldots,i_{T}$, a sequence of vertices $v_1,\ldots,v_{T}$ and a sequence of path forests $\Path_1,\ldots,\Path_{T+1}$.
			These indices, vertices and path forests will satisfy the following properties.
			\begin{enumerate}[(a)]
				\item\label{b}
				We have $i_{k}>i_{k+1}>0$ for $1\leq k < T$.
				\item\label{a} The edge added to $\Path_k$ and the edge deleted from $\Path_k$ at the $k$-th step to obtain $\Path_{k+1}$ have the same colour for $k<T$. In particular, $C(\Path_{k+1})=C(\Path_k)$.
				\item\label{e} The edge $e$ added at the $k$-th step to $\Path_{k}$ goes from $v_k$ to $Q_{i_k-1}$.
				\item\label{c} For $k=1,\ldots, T$ we have $v_k\in l(\Path_{k})$.
				\item\label{d} For $k\geq 2$ we have $v_k\in V_{i_k}\cap V(\Path)\setminus l(\Path)$. For $k=1$ we have $v_k \in V_{i_k}$.
				
				\item\label{f} The edge $e$ added at the $k$-th step to $\Path_{k}$ satisfies $c(e)\in C_{i_k-1}$ as well as $c(e)\in C_{i_{k+1}}\setminus C_{i_{k+1}-1}$.			
			\end{enumerate}
			It is easily checked that these properties hold for $i_1,v_1,\Path_1$ as defined above and we will show that they still hold after the $k$-th step for $i_{k+1}, v_{k+1},\Path_{k+1}$. 
			We list three easy consequences of the properties \eqref{a}-\eqref{f} above.
			\begin{Claim}\leavevmode
				\begin{enumerate}[(i)]						
					\item\label{i} The algorithm terminates after a finite number of steps.
					\item\label{ii} We have $Q_{i_k-1}\subset f(\Path_{k})$.
					\item\label{iii} Before the $k$-th step, no edge $e$ of colour $c(e)\in C_{i_k-1}$ was deleted (from $\Path_i$ for any $i$ with $1\leq i< k$).
				\end{enumerate}
			\end{Claim}
			\begin{proof}
				For \eqref{i}, note that by property \eqref{b} the sequence $i_1,i_2,\ldots$ is strictly decreasing and bounded below by 0. Since all the indices are integers, the sequence must be finite.
				
				For \eqref{ii}, note that in the $k$-th step of the algorithm, we add an edge incident to $Q_{i_k-1}$ by \eqref{e}. In particular, since $i_1>i_2>\ldots$, this implies that before the $k$-th step, we added no edge directed to $Q_{i_{k}-1}$. This implies that each vertex in $Q_{i_k-1}$ has in-degree 0 in the graph $\Path_{k}$ or equivalently $Q_{i_k-1}\subset f(\Path_{k})$.
				
				For \eqref{iii}, note that the edge deleted at the $k$-th step has a colour lying in $C_{i_{k+1}}\setminus C_{i_{k+1}-1}$ by \eqref{a} and \eqref{f}. Since $i_1>i_2>\ldots$ and $C_0\subset C_1\subset C_2\subset\ldots$, this implies that before the $k$-th step, no edge $e$ of colour $c(e)\in C_{i_k-1}$ was deleted.	
			\end{proof}
			
			We now describe the $k$-th step of the algorithm. Recall that it has input $i_k,v_k,\Path_{k}$ satisfying properties \eqref{a}-\eqref{f}. We need to show that its output $i_{k+1},v_{k+1},\Path_{k+1}$ also satisfies these properties (for $k<T$). 
			
			\noindent\underline{$k$-th step of algorithm:} 
			\begin{itemize} 
				\item We first explain which edge will be added to $\Path_{k}$. By \eqref{d} we have $v_{k}\in V_{i_k}$, so there are at least two edges from $v_{k}$ to $Q_{i_k-1}$ using colours in $C_{i_k-1}$ (by the definition of $V_{i_k}$). Call these edges $e=v_k f$ and $e'=v_kf'$. By \eqref{c} and \eqref{ii} we have $v_k\in l(\Path_{k})$ and $f,f'\in Q_{i_k-1}\subset f(\Path_{k})$. Thus, by Observation \ref{add} either $\Path_{k}+e$ or  $\Path_{k}+e'$ is a path forest. Assume without loss of generality that $\Path_{k}+e$ is a path forest. $e$ is the edge that we will add to $\Path_k$. Hence, \eqref{e} will be true after the $k$-th step. Note also that $c(e)\in C_{i_k-1}$, so that the first part of \eqref{f} is satisfied. 
				
				\item Since we want to add $e$ to $\Path_{k}$, we might need to delete an edge from $\Path_{k}$ which has colour $c(e)$ to make sure that $\Path_{k+1}$ will be rainbow. We distinguish two cases. \\
				\textbf{Case 1:} If  $c(e)\in C_0$, then define $\Path_{k+1}=\Path_{k}+e$. Observe that $\Path_{k+1}$ is rainbow, since $C(\Path_{k})=C(\Path_{0})$ (by \eqref{a}) does not contain $c(e)$ before adding $e$. Terminate the algorithm. \\
				\textbf{Case 2:} If $c(e)\notin C_0$, then choose $i_{k+1}$ so that $c(e)\in C_{i_{k+1}}\setminus C_{i_{k+1}-1}$. This is possible since $C_0\subset C_1\subset C_2 \subset\ldots$ by construction. Note that the second part of \eqref{f} is now also satisfied. Recall that $c(e) \in C_{i_k-1}$, so that $0< i_{k+1}\leq i_k-1$, assuring that \eqref{b} is true. By Observation \ref{newcolour}, we can find $v_{k+1}\in V_{i_{k+1}}\cap V(\Path)\setminus l(\Path)$ with $c(v_{k+1})=c(e)$. Define $\Path_{k+1}=\Path_{k}+e-v_{k+1}v^+_{k+1}$, so that \eqref{a} holds.\\				
				This step implicitly assumes that $v_{k+1}v^+_{k+1}\in\Path_{k}$. To see that this is true, note that $v_{k+1}v^+_{k+1}\in \Path_1$ and that, by \eqref{iii}, it has not been deleted from $\Path_1$ before the $k$-th step, so that it must still be in $\Path_{k}$.		 
				
				\item We have already seen that properties \eqref{a},\eqref{b},\eqref{e},\eqref{f} hold for this choice of $i_{k+1},v_{k+1},\Path_{k+1}$.  \eqref{c} is satisfied since we deleted $v_{k+1}v_{k+1}^+$ from $\Path_{k}$ to obtain $\Path_{k+1}$. \eqref{d} holds by our choice of $v_{k+1}$.  
				
				\item Continue with step $k+1$.
			\end{itemize}			
			Using Observation \ref{pathsize}, it is now easy to check that $|\Path_{i+1}|=|\Path_{i}|$ for $i=1,\ldots, T-1$ and $|\Path_{T+1}|= |\Path_{T}|-1$. So, all in all, we get $|\Path_{T+1}|= |\Path_1|-1 \leq \gamma n$. It is also not hard to check that $e(\Path_1)=\ldots=e(\Path_{T})=e(\Path_{T+1})-1$ so that $\Path_{T+1}$ is a rainbow path forest with at most $\gamma n$ paths and one more edge than $\Path_1=\Path$, contradicting maximality of $\Path$ and thus finishing the proof of the claim.
			
			It remains the case $v_1\in V(G)\setminus V(\Path)$ (as opposed to the case $v_1\in l(\Path)$). In this case, property \eqref{c} is violated in the first step of our algorithm. We can carry it out nevertheless, adding an edge from $v_1$ to $Q_{i_1-1}$ and deleting an edge of the same colour to obtain $\Path_2$ from $\Path_1$. We then have $|\Path_2|=|\Path_1|+1$. From now on, the algorithm works exactly as before and we have $|\Path_{i+1}|=|\Path_{i}|$ for $i=2,\ldots, T-1$ and $|\Path_{T+1}|= |\Path_{T}|-1$. So, all in all, we get $|\Path_{T+1}|= |\Path_1| \leq \gamma n$ as before, finishing the proof.
		\end{proof}
		We now use the claim to prove $|C_i|\geq |C_{i-1}|+ \delta n$ as promised. We have $|C_i|\geq |C_0|+|V_i\cap V(\Path)\setminus l(\Path)|$, since for each $v\in V_i\cap V(\Path)\setminus l(\Path)$, we have $c(v)\in C_i$. Using the claim we get $|V_i\cap V(\Path)\setminus l(\Path)| = |V_i|$. Combining these two inequalities we obtain: 
		\begin{equation}\label{ci}
		|C_i|\geq|C_0|+|V_i|.
		\end{equation} 
		We now use a double counting argument to bound $|V_i|$ in terms of $|C_{i-1}|-|C_0|$. Choose $A\in \R$ so that $|C_{i-1}|-|C_0|=A\delta n$. Observe that for any vertex $v$ we have $d^-(v,C_{i-1})\geq d^-(v)-(|C(G)|-|C_{i-1}|)\geq (1-\delta)n-(|C(G)|-|C_0|) + (|C_{i-1}|-|C_0|)\geq 2\delta n + A\delta n $. Here, in the last inequality we used $|C(G)|-|C_0|=|C(\Path)|\leq (1-3 \delta) n$. Using  $d^-(v,C_{i-1})\geq (A+2)\delta n$ we obtain
		\begin{equation}\label{deg}
		d^-(Q_{i-1},C_{i-1})\geq (A+2)\delta n \cdot |Q_{i-1}|= (A+2) n.
		\end{equation}
		On the other hand, writing $X=\{v\in V(G): d(v,Q_{i-1},C_{i-1})\leq 1\}$ and $V_i=\{v\in V(G): d(v,Q_{i-1},C_{i-1})> 1\}$, we also get 
		\begin{equation}\label{count}
		d^-(Q_{i-1},C_{i-1})\leq |X| + |V_i|\cdot |Q_{i-1}| \leq n + |V_i|/\delta.
		\end{equation}
		Here, the first inequality uses $d(v,Q_{i-1},C_{i-1})\leq |Q_{i-1}|$ for $v\in V_i$. Putting \eqref{deg} and \eqref{count} together, we finally obtain $|V_i|\geq (A+1)\delta n$. Using this in \eqref{ci} we get 
		$$
		|C_i|-|C_0|\geq (A+1)\delta n.
		$$
		Recalling the definition of $A$, this establishes $|C_i|\geq |C_{i-1}|+\delta n$ as desired.
	\end{proof}

	\subsection{Rotating and gluing the path forest}
	The following Lemma shows how to `rotate' and `glue together' a directed rainbow path forest. In \cite{longcycle} an undirected version of this Lemma was proved using the concept of path rotations. We use methods for rotating directed paths (for example used by Frieze and Krivelevich \cite{FK}) and some new ideas needed to produce a rainbow structure. 
	\begin{Lemma}\label{glue}
		Let $b,m,r\in \N$ so that $mr\leq b$. Let $\Path =\{P_1,\ldots, P_r\}$ be a rainbow path forest with $V(\Path)\subset V(G)$ and let $H$ be a properly coloured digraph,  with $V(H)=V(G)$, so that $C(H)$ is disjoint from $C(\Path)$ and so that for any two disjoint sets of vertices $A,B\subset V(G)$ with $|A|=|B|=b$ we have $d_H(A,B)\geq 2b+1$. Also assume that $H$ has minimum in-degree $\delta_H^-\geq 5b$. Then either $|P_1|\geq v(\Path)-2b$ or there exist edges $e_1,e_2,e_3\in H$ and a rainbow path forest $\Path'=\{P_1',\ldots,P_r'\}$ such that $E(P')\subset E(\Path) +e_1+e_2+e_3$ and $V(\Path')\supset V(\Path)$ and $|P_1'|\geq |P_1|+m$.
	\end{Lemma}
	\begin{proof}	
		Assume that $|P_1| < v(\Path) - 2b$. For $i=1,\ldots,r$ write $P_i = v_{i,1}\rightarrow\ldots\rightarrow v_{i,|P_i|}$ and define $v_k=v_{1,k}$ for the vertices on $P_1$. Set 
		\begin{equation*}
		\begin{aligned}
		A&=\{v_{j}: 2\leq j \leq 2b\}\qquad &A'&=\{v_{j}:  j > 2b\}\\
		B&=\{v_{i,j}:\;i\geq 2, j < m\} \qquad &B'&=\{v_{i,j}:\;i\geq 2, j \geq  m\}.
		\end{aligned} 
		\end{equation*}
		Note that $2b\leq v(\Path)- |P_1| = |B|+|B'| \leq mr + |B'|$ which implies $|B'|\geq b$ since $mr\leq b$. In the following let $B_0\subset B'$ be a subset of size $|B_0|=b$. 
		
		Observe that in $H$ there are at least $5b$ edges directed to $v_1$. They start at $A\cup B$ or at $A'$ or at $B'$ or at $V(G)\setminus V(\Path)$. Observe that $|A\cup B|\leq 3b$  so that there are at least $2b$ edges from $A'\cup B'\cup \left(V(G)\setminus V(\Path)\right)$ to $v_1$. We distinguish three cases: Either there is an edge from $B'$ to $v_1$ or there are $\geq b$ edges from $V(G)\setminus V(\Path)$ to $v_1$ or there are $>b$ edges from $A'$ to $v_1$. Figure \ref{rotate} illustrates how the path forest $\Path$ changes in each of the three cases as described below.
		
		\underline{Case 1:} If $H$ contains an edge $e=v_{s,t}v_{1}$ from $B'$ to $v_{1}$ (i.e. $s\geq 2$ and $t\geq m$), then divide $P_s$ into two parts $Q_1$ from $v_{s,1}$ to $v_{s,t}$ and $Q_2$ from $v_{s,t+1}$ to $v_{s, |P_s|}$. Note that $Q_2$ might be empty. Now define $P_1'=Q_1+e+P_1$ and $P_s'=Q_2$ as well as $P_i'=P_i$ for $i\neq 1,s$. Then $\Path'=\{P_1',\ldots,P_r'\}$ satisfies the claim of the Lemma for $e_1=e_2 = e_3=e$.
		
		\underline{Case 2:} Assume now that there are $\geq b$ edges from $V(G)\setminus V(\Path)$ to $v_1$. Choose a subset  $X \subset N_H^-(v_{1})\cap V(G)\setminus V(\Path)$ of size $|X|=b$. Recall $B_0\subset B'$ with $|B_0|=b$. Since $d_H(B_0,X)\geq 2b+1$, there is a vertex $x\in X$ with $d_H(B_0,x)\geq 2$. Pick two edges $f,f'\in H$ from $B_0$ to $x$, they have distinct colours since $H$ is properly coloured. So we may without loss of generality assume that $c(f)\neq c(xv_{1})$. Write $f=v_{s,t}x$ and define $Q_1$ and $Q_2$ as before. Set $P_1'= Q_1 + f + xv_{1} + P_1$ and $P_s'=Q_2$ as well as $P_i'=P_i$ for $i\neq 1,s$. Now $\Path'=\{P_1',\ldots,P_r'\}$ with $e_1=e_2=f$ and $e_3 = xv_{1}$ satisfies the claim.
		
		\underline{Case 3:} We now consider the case $d_H(A',v_{1})> b$. We call a vertex $v_k\in A$ good if there is an index $a$ with $2b < a < |P_1|$ so that the edges $v_{a}v_1, v_{k-1}v_{a+1}$ both lie in $H$ and have different colours. In this case we call $v_a$ a friend of $v_k$. If $v_k$ is good, then there is a `rotation' of $P_1$ that starts at $v_k$: Define $R_1=v_1\rightarrow v_2\rightarrow\ldots\rightarrow v_{k-1}$ and $R_2=v_k\rightarrow v_{k+1}\rightarrow\ldots\rightarrow v_{a}$ and $R_3=v_{a+1}\rightarrow v_{a+2}\rightarrow\ldots\rightarrow v_{|P_1|}$ and observe that the path  $P_1(v_k,v_a):= R_2 + v_av_1 + R_1 + v_{k-1}v_{a+1} + R_3$ is a rainbow path covering exactly the vertices of $P_1$ and having $v_k$ as a starting point.
		
		We now claim that there are at least $b$ good vertices. If not, pick a set $X\subset A$ of vertices that are not good and has size $|X|=b$. Also pick a subset $Y\subset N_H^-(v_1)\cap A'\setminus\{v_{|P_1|}\}$ of size $|Y|=b$ (this is possible since $d_H(A',v_{1})> b$). We have $d_H(X^-,Y^+)\geq 2b+1$, so we can pick a vertex $v_{a+1}\in Y^+$ so that there are two edges $e=v_{k-1}v_{a+1}, e'=v_{j-1}v_{a+1}$ in $H$ with $v_{k-1},v_{j-1}\in X^-$. Since $c(e)\neq c(e')$ we may without loss of generality assume that $c(e)\neq c(v_{a}v_1)$. But this shows that $v_{k}\in X$ is good, contradicting our choice of $X$. We have thus established that there are at least $b$ good vertices in $A$. 
		
		Let $X\subset A$ be a set of good vertices of size $|X|=b$. Since $d_H(B_0,X)\geq 2b+1$, there is a vertex $v_k\in X$ so that $d_H(B_0,v_k)\geq 3$. Pick three edges $f_1,f_2,f_3\in H$ that go from $B_0$ to $v_k$. Since $v_k$ is good, it has a friend $v_a$. Since the colours of $f_1,f_2,f_3$ are pairwise distinct, we may assume without loss of generality that $c(f_1)\neq c(v_{k-1}v_{a+1}),c(v_{a}v_1)$. Write $f_1 = v_{s,t}v_k$, define $Q_1, Q_2$ as before.  Finally, set $P_1'=Q_1+f_1+ P_1(v_k,v_a)$
		and $P_s'=Q_2$ as well as $P_i' = P_i$ for $i\neq 1,s$. Again, $\Path'=\{P_1,\ldots,P_r'\}$ satisfies the claim for $e_1=v_{k-1}v_{a+1},e_2=v_{a}v_1,e_3=f_1$.
			\begin{figure}
				\centering
				\begin{tikzpicture}[inner sep=1,
				vertex/.style={circle,draw,fill,}
				]
				\begin{scope}[very thick,decoration={
					markings,
					mark=at position 0.5 with {\arrow{>}}}
				] 
				\node[vertex,label=below left:$v_1$] (v1) at (0,0) {};
				\node[vertex] (vend) at (0,2.5) {};
				\draw[postaction={decorate}] (v1)--(vend);
				
				\node[vertex,label=below right:$v_{s,1}$] (vs1) at (1,0) {};
				\node[vertex,label=right:$v_{s,t}$] (vst) at (1,1.2) {};
				\node[vertex,label=right:$v_{s,t+1}$] (vst1) at (1,1.8) {};
				\node[vertex] (vse) at (1,2.5) {};
				\draw[postaction={decorate}] (vs1)--(vst);
				\draw[postaction={decorate}] (vst)[gray, very thin]--(vst1);
				\draw[postaction={decorate}] (vst1)--(vse);
				\draw[postaction={decorate}, densely dotted] (vst)--(v1);
				
				\node at (0.5,-1) {Case 1};
				
				\node[vertex,label=below left:$v_1$] (w1) at (3.7,0) {};
				\node[vertex] (wend) at (3.7,2.5) {};
				\draw[postaction={decorate}] (w1)--(wend);
				
				\node[vertex,label=below right:$v_{s,1}$] (ws1) at (5.3,0) {};
				\node[vertex,label=right:$v_{s,t}$] (wst) at (5.3,1.2) {};
				\node[vertex,label=right:$v_{s,t+1}$] (wst1) at (5.3,1.8) {};
				\node[vertex] (wse) at (5.3,2.5) {};
				\draw[postaction={decorate}] (ws1)--(wst);
				\draw[postaction={decorate}] (wst)[gray, very thin]--(wst1);
				\draw[postaction={decorate}] (wst1)--(wse);

				\node[vertex,label=right:$x$] (con) at (4.5,3.5) {}; 
				\draw (con) ellipse (1 and 0.25) node[above=0.27] {$V(G)\setminus V(\mathcal{P})$};
				\draw[postaction={decorate},densely dotted] (wst)--(con);
				\draw[postaction={decorate},densely dotted] (con)--(w1);
				\node at (4.7,-1) {Case 2};
				
				\node[vertex,label=below:$v_1$] (x1) at (8,0) {};
				\node[vertex,label=below left:$v_{k-1}$] (xk1) at (8,1) {};
				\node[vertex,label=above right:$v_k$] (xk) at (8,1.5) {};
				\node[vertex,label=right:$v_a$] (xa) at (8,2.5) {};
				\node[vertex,label=right:$v_{a+1}$] (xa1) at (8,3) {};
				\node[vertex] (xe) at (8,4) {};
				\draw[postaction={decorate}] (x1)--(xk1);
				\draw[postaction={decorate}] (xk1)[gray, very thin]--(xk);
				\draw[postaction={decorate}] (xk)--(xa);
				\draw[postaction={decorate}] (xa)[gray,very thin]--(xa1);
				\draw[postaction={decorate}] (xa1)--(xe);
				\draw[postaction={decorate},densely dotted] (xa).. controls (8.7,1.5) ..(x1);
				\draw[postaction={decorate},densely dotted] (xk1).. controls (7.3,2) ..(xa1);

				\node[vertex,label=below right:$v_{s,1}$] (xs1) at (9.5,0) {};
				\node[vertex,label=right:$v_{s,t}$] (xst) at (9.5,1.2) {};
				\node[vertex,label=right:$v_{s,t+1}$] (xst1) at (9.5,1.8) {};
				\node[vertex] (xse) at (9.5,2.5) {};
				\draw[postaction={decorate}] (xs1)--(xst);
				\draw[postaction={decorate}] (xst)[gray, very thin]--(xst1);
				\draw[postaction={decorate}] (xst1)--(xse);
				\draw[postaction={decorate},densely dotted] (xst).. controls (8.9,1.5) ..(xk);
				
				\node at (8.9, -1) {Case 3};
				
				\end{scope}
				\end{tikzpicture}
				\caption{The three different cases that can arise in the proof of Lemma \ref{glue}. In each case, only the two paths that change, namely $P_1$ and $P_s$, are displayed. The edges that will be deleted from the path forest are thin and gray, the edges from the expander $H$ that will be added, i.e. $e_1,e_2,e_3$, are dotted.}\label{rotate}
			\end{figure}
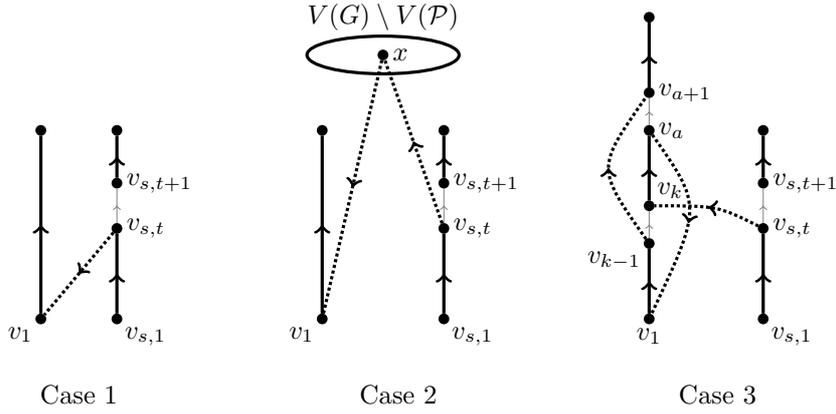		
	\end{proof}

	\subsection{Proof of Theorem \ref{TLRC}}
	We have now collected all the ingredients to prove Theorem \ref{TLRC} and show how to combine them in this section. 	
	We start by noting that the existence of long rainbow path forests follows directly from Lemma \ref{long} applied to a properly edge-coloured $\overleftrightarrow{K_n}$. Take e.g. $\delta = 2n^{-1/3}$ and $\gamma = n^{-1/3}$. We now turn our attention to the existence of long rainbow cycles.
	
	Consider a proper colouring of $\overleftrightarrow{K_n}$. Create a subgraph $H$ of $\overleftrightarrow{K_n}$ as in Theorem \ref{rand} with $p=6n^{-1/5}$. Set $G=K_n\setminus H$. Let $b=n^{4/5}$. With high probability every $v\in V(G)$ satisfies $5.5 b\leq d^-_H(v)\leq 7b-1$. Thus, any $v\in G$ satisfies $d^-_G(v)=n-1-d^-_H(v)\geq (1-7n^{-1/5})n$. By Lemma \ref{long}, for $\gamma = n^{-3/5}$ and $\delta = 7n^{-1/5}$, $G$ contains a rainbow path forest $\Path=\{P_1,\ldots,P_r\}$ with $r\leq\gamma n=n^{2/5}$ and $e(\Path)\geq(1-21n^{-1/5})n$. 
		
	We now apply Lemma \ref{glue} repeatedly to $\Path$ with the expander $H$ constructed above, recall $b=n^{4/5}$, $r\leq n^{2/5}$ and set $m=n^{2/5}$. All in all, we apply the Lemma $n^{3/5}$ times. After each application, we delete all the edges from $H$ that have the same colour as $e_1,e_2$ or $e_3$. After $i\leq n^{3/5}$ applications of Lemma \ref{glue} each vertex  $v\in H$ has in-degree at least $5.5n^{4/5}-3i\geq 5.5n^{4/5}-3n^{3/5}\geq 5 b$. Moreover, for any two disjoint sets $A,B\subset V(G)$ of size $|A|=|B|=b$, we have, after the $i$-th application, $d_H(A,B)\geq(1-o(1))b^2p-3bi\geq 2n^{7/5}\geq 2b+1$. So we actually may apply Lemma \ref{glue} $n^{3/5}$ times. At each application, we have either that already $|P_1|\geq V(\Path)-2b\geq (1-23n^{-1/5})n$ or that the length of $P_1$ increases by $m$. Since $m\cdot n^{3/5}=n$ we must have the first of these two cases before the last iteration. We have thus constructed a rainbow path $P_1$ with $|P_1|\geq(1-23n^{-1/5})n$. In order to turn $P_1$ into a cycle, observe that between the last $b$ and the first $b$ vertices on $P_1$ there is an edge that lies in $H$ (or rather what is left of $H$). We thus obtain a cycle with at least $(1-25n^{-1/5})n$ vertices as desired.

	\section{Rainbow Trees}\label{RT}
	In this section we prove Theorem \ref{TRT}.	We start with a brief outline of the proof. Let $T$ be a tree on $n$ vertices with maximum degree $\Delta(T)\leq \maxd n/ \log n$ and let $c$ be a globally $\glo n$-bounded, proper colouring of $K_n$. We will find a rainbow embedding of $T$ into $K_n$ in two steps. Denote by $R$ the $n/4$ vertices of $T$ with highest degrees and set $W=V(T)\setminus R$. In the first step we will find a rainbow embedding of $R$ into $K_n$ in a random-greedy fashion. More precisely, we will find a subset $B\subset V(K_n)$ and a bijection $\pi : R\rightarrow B$ so that $\pi\left(T[R]\right)=:T_0$ is rainbow. We also write $A=V(K_n)\setminus B$. We will make sure that this partial embedding $\pi$ of $T$ will satisfy sufficiently nice properties, so that we can carry out the second step of our embedding: We will show that a uniformly random bijection $\tau : W\rightarrow A$ gives rise to a rainbow embedding of $T$. More precisely, the embedded tree $(\pi\cup\tau)(T)$ will be rainbow with positive probability. For this part our main tool is a version of the Local Lemma due to Lu and Sz\'ekely in \cite{LS}.

	Before giving precise descriptions of these two steps in the next sections we explain what the `nice' properties of the partial embedding $\pi$ are and give some intuition why they help us in the second step. To this end we need some notation.
	For an edge $e\in K_n$ denote by $c(e)$ its colour and for a subgraph $H$ of $K_n$ denote by $c(H)$ the set of all colours appearing in $H$. For a colour $f\in c(K_n)$ and a vertex $a\in K_n$ we say that $f$ is present at $a$ if there is an edge incident to $a$ which has colour $f$.
	
	Recall that we aim for a `nice'  rainbow embedding of $T[R]$, i.e. a bijection $\pi : R\rightarrow B$ so that $\pi\left(T[R]\right)=:T_0$ is rainbow. 
	For such a embedding we define values $w_f, w_g, m_g$ for each $f\in c(K_n)$ and each $g\in V(K_n)$ below. The embedding $\pi$ will be `nice' if all these quantities are small. 
	
	For $f\in c(K_n)$, we define the weight $w_f$ of $f$ to be
	\begin{eqnarray*}
		w_f = \sum_{\substack{b\in B \\ f\text{ present at } b}} d_T\left(\pi^{-1}(b)\right).
	\end{eqnarray*}
	Intuitively, this is a rough measure of how likely the colour $f$ is to be used by the embedding $(\pi\cup\tau)$ (where $\tau$ is a uniformly random bijection from $W$ to $A$). If $f$ is present at some vertex $b\in B$, so that $\pi^{-1}(b)$ has a high degree in $T$, then it is more likely to be used in the embedding $\pi\cup\tau$.
	
	For $g\in V(K_n)$, set $w_g = 0$ if $g\in B$. Otherwise $g\in A$ and we let
	\begin{eqnarray*}
		w_g = \sum_{\substack{u\in R \\c(g\pi(u))\in c(T_0) } } d_T(u).
	\end{eqnarray*}
	This quantity roughly measures (in fact it uses a union bound) how likely we are to embed an `unsuitable' vertex $w\in W$ onto $\tau(w)=g$. $w$ is `unsuitable' for $g$, if there is a vertex $u\in R$ so that $uw\in T$ and so that the colour of the edge $\pi(u)g$ was already used in $c(T_0)$.
	
	Finally, for $g\in V(K_n)$ we define $m_g$ to be  
	\begin{equation*}
	m_g = \left\vert\{a\in A :c(ag)\in c(T_0)\}\right\vert.
	\end{equation*}
	This value simply counts how many colours incident to $g$ were already used in $T_0$ and therefore should not be used again when using $\tau$ to extend the embedding.

	We now state our two main Lemmas describing the results of the first and second step of our embedding.
	\begin{Lemma}\label{RGEL}
		Let $T$ be a tree on $n$ vertices with maximum degree at most $\maxd n/\log n$ and let $c$ be a proper, globally $\glo n$-bounded edge colouring of $K_n$.\\
		Let $\aon,\atw, \ath > 0$ satisfy the following relations: 
		\begin{equation}\label{eq1}
		16\glo \leq \aon,\quad 16\atw\leq \ath, \quad 32\aon\leq \ath\quad \text{ and }\quad \aon^2>32\maxd,\quad  \atw^2>64\maxd,\quad \ath^2> 64\maxd.
		\end{equation}
		Assume moreover, that $\bad,\nbad>0$ satisfy 
		\begin{equation}\label{eq2}
		\bad\nbad\geq 2\glo, \quad 16\nbad \leq \atw, \quad 2\bad \leq \atw.
		\end{equation}
		Then there is a rainbow embedding $\pi$ of the $n/4$ vertices of $T$ which have highest degrees, so that the partial embedding $\pi$  satisfies $w_f\leq \aon n$, $m_g\leq \atw n$ and $w_g\leq \ath n$ for all colours $f\in c(K_n)$ and all vertices $g\in V(K_n)$.
	\end{Lemma}
	
	\begin{Lemma}\label{LLLS}
		Let $0<\roots,\aio,\glo< 1\leq \degree$ so that 
		\begin{equation}\label{sum2}
		16(1-r)^{-3}	\cdot \Bigl(\aio(3D + 11) + \glo(16D + 12)\Bigr)\leq 1.
		\end{equation}
		Assume that $c$ is a proper colouring of $K_n$ which is globally $\glo n$-bounded. Let $T$ be a tree on $n$ vertices, which are partitioned into $V(T)=W\cup R$. Assume also that $\pi : R\rightarrow B\subset V(K_n)$ is a rainbow embedding of $T[R]$ and write $A=V(K_n)\setminus B$.\\ 		
		Suppose that 
		$$|R|\leq \roots n \qquad\text{ and }\qquad \forall w\in W:\:d_T(w)\leq D.$$ 
		Further suppose that for all $g\in V(K_n)$ and for all $f\in c(K_n)$, we have 
		$$w_g \leq \aio n,\quad m_g\leq \aio n,\quad  w_f\leq \aio n.$$ 
		Then, for $n$ sufficiently large, there is a bijection $\sigma : W\rightarrow A$ extending $\pi$ so that $(\pi\cup\sigma)(T)$ is rainbow.
	\end{Lemma}

	In Section \ref{tools} we give the precise statements of our two main tools, namely the Azuma-Hoeffding concentration inequality and the Local Lemma in the framework of Lu and Sz\'ekely. In Sections \ref{RGE} and \ref{LLL} we prove the Lemmas \ref{RGEL} and \ref{LLLS} respectively. We combine these two Lemmas to give the proof of Theorem \ref{TRT} in Section \ref{TRP}.

	\subsection{Tools}\label{tools}
	In Section \ref{RGE} we will be concerned with certain random variables and showing that they are concentrated around their expected value. To this end we will use the well known Azuma-Hoeffding inequality due to Azuma \cite{Azuma} and Hoeffding \cite{Hoef}. We will use the following special case: 
	\begin{Theorem}\emph{(Azuma-Hoeffding inequality)}
		Suppose $Y_1,\ldots,Y_n$ are random variables, so that $Y_i$ takes the value $0$ or $d_i$ and so that $Y_i$ is mutually independent of $Y_1,\ldots, Y_{i-1}$. Then, writing $Y=Y_1+\ldots+Y_n$ and $\sigma = \sum_{i=1}^n d_i^2$, we have for any $t>0$ that 
		\begin{equation*}
		\Pr\bigl[Y-\E[Y]>t\bigr]\leq e^{-t^2/2\sigma}.
		\end{equation*}
	\end{Theorem}
	In Section \ref{LLL} we will make use of the Lov\'asz Local Lemma. It is a powerful tool for probabilistic existence proofs established by Erd\H{o}s and Lov\'asz \cite{orLLL}. 
	Erd\H{o}s and Spencer observed that the statement of the Local Lemma could be generalised and formulated the so called \textit{Lopsided Local Lemma} \cite{LopLLL}, see e.g. \cite[Chapter 19.3]{MolReed} for a precise statement and proof. The work of Lu and Sz\'ekely \cite{LS} makes use of the Lopsided Local Lemma by finding a \textit{negative dependency digraph} of a certain set of events in the probability space of random injections. A combination of their main result in \cite{LS} and the Lopsided Local Lemma gives Theorem \ref{LocalS} below. To make its statement precise, we start with some definitions.	
	
	Suppose we are given finite sets $X$ and $Y$ of the same cardinality and consider the probability space $\Omega$ given by picking a bijection $\sigma:X\rightarrow Y$ uniformly at random from all such bijections, denote the set of these bijections by $\mathcal{S}$. Let $\tau : T\rightarrow U$ be a given bijection between two sets $T\subset X$ and $U\subset Y$. The corresponding \textit{canonical event} $\Omega(T,U,\tau)$ consists of all bijections $\sigma: X\rightarrow Y$ extending $\tau$, i.e.
	\begin{equation*}
	\Omega(T,U,\tau) =\{\sigma\in \mathcal{S} : \sigma(x)=\tau(x)\text{ for all }x\in T\}.
	\end{equation*}
	We say that two events $\Omega(T_1,U_1,\tau_1)$ and $\Omega(T_2,U_2,\tau_2)$ $\mathcal{S}$-\textit{intersect} if $T_1$ and $T_2$ intersect or $U_1$ and $U_2$ intersect and write $\Omega(T_1,U_1,\tau_1)\sim \Omega(T_2,U_2,\tau_2)$. 
	\begin{Theorem}\label{LocalS}
		\emph{(Asymmetric Lopsided Lov\'asz Local Lemma in the framework of Lu and Sz\'ekely \cite{LS})} 		
		Let $\mathcal{B}$ be a collection of canonical events. Then, with positive probability none of the events $B\in\mathcal{B}$ occurs, provided that for all $B\in \mathcal{B}$ it holds that
		\begin{equation*}
		\sum_{\substack{B'\in\mathcal B\\ B'\sim B}} \Pr[B'] \leq \frac{1}{4}.
		\end{equation*}
		
	\end{Theorem}

	\subsection{Proof of Lemma \ref{RGEL}}\label{RGE}
	\begin{proof}
		The overall strategy of our proof is simple. We will embed the $n/4$ high degree vertices in a random-greedy fashion and show that the desired properties hold with high probability. We first describe the random-greedy embedding, which we will refer to as RGE from now on. 
		
		Denote the vertices of $T$ by $v_1,\ldots,v_n$ so that (i) $v_1,\ldots,v_{n/4}$ are those vertices of $T$ with the highest degrees (splitting ties arbitrarily) and so that (ii) for each $i\in[n/4]$ there is at most one $1\leq j<i$ so that $v_j v_i$ is an edge of $T$. Condition (ii) can be satisfied since the $n/4$ vertices with highest degrees of $T$ span a forest, which is $1$-degenerate. Sticking with the notation introduced previously, we set $R = \{v_1,\ldots,v_{n/4}\}$, $W=\{v_{n/4+1},\ldots,v_n\}$. 
		
		Moreover, write $d_i = d_T(v_i)$ for $i=1,\ldots,n$. For distinct vertices $b_1,\ldots,b_k \in K_n$ we denote by $T[b_1,\ldots,b_k]$ the subgraph of $K_n$ in which $b_ib_j$ is an edge precisely if $v_iv_j$ is an edge of $T$. 
		
		We now describe the random-greedy embedding: We will perform $n/4$ steps of the following form. For the $k$-th step, suppose that we have already chosen distinct $b_1,\ldots,b_{k-1}\in K_n$ so that $T[b_1,\ldots,b_{k-1}]$ is rainbow. We say that a vertex $b\in K_n$ is feasible if (i) $b\notin\{b_1,\ldots,b_{k-1}\}$ and (ii) the subgraph $T[b_1,\ldots,b_{k-1},b]$ is rainbow. Out of the feasible vertices we choose one vertex $b_k$ uniformly at random (and independently of previous choices) and then continue with the $(k+1)$-th step. It is clear that, if successful, RGE produces a rainbow subgraph $T[b_1,\ldots,b_{n/4}]$. We will not only show that RGE is successful, but that in each step, there are at least $n/2$ feasible vertices, making the embedding sufficiently random for our purposes.
		
		\begin{Claim}
			At each of its $n/4$ steps, RGE has at least $n/2$ feasible choices.
		\end{Claim}
		\begin{proof}
			At the $k$-th step of RGE, there are at most $k-1\leq n/4$ vertices $b$ that violate (i) $b\notin\{b_1,\ldots,b_{k-1}\}$. To bound how many vertices $b$ violate (ii) $T[b_1,\ldots,b_{k-1},b]$ being rainbow, observe the following: $v_k$ has at most one neighbor $v_j$ in $\{v_1,\ldots,v_{k-1}\}$. If $v_k$ has no such neighbor, then any vertex $b\in K_n\setminus\{b_1,\ldots,b_{k-1}\}$ will make $T[b_1,\ldots,b_{k-1},b]$ rainbow since $T[b_1,\ldots,b_{k-1},b]$ consists of the same set of edges as $T[b_1,\ldots,b_{k-1}]$. Otherwise, let $v_j$ be the unique neighbor of $v_k$ with $j<k$. Then the only vertices $b$ that will not lead to a rainbow copy of $T[b_1,\ldots,b_{k-1},b]$ are those with $c(b_j b)\in c(T[b_1,\ldots,b_{k-1}])$. Since the colouring $c$ is proper, all colours present at $b_j$ are distinct and there are at most $|c(T[b_1,\ldots,b_{k-1}])|\leq e(T[b_1,\ldots,b_{k-1}])\leq n/4$  vertices $b$ with $c(b_jb)\in c(T[b_1,\ldots,b_{k-1}])$. So all in all, there are at most $n/4+n/4 = n/2$ vertices in $K_n$ which are not feasible, leaving at least $n/2$ feasible vertices. 
		\end{proof}
		
		This finishes the description of RGE and we are now ready to prove Lemma \ref{RGEL}. We need to prove the three inequalities $w_f\leq \aon n$, $m_a\leq \atw n$ and $w_a\leq \ath n$. They will be derived in a standard way: We first bound the expectation of the given random variable, prove tight concentration using the Azuma-Hoeffding inequality and finally apply a union bound over all colours/vertices. This is where we will use our conditions that the colouring $c$ is proper and globally $\glo n$-bounded as well as the assumption that $\Delta(T)\leq \maxd n/ \log n$. Observe that the latter means $d_i\leq \maxd n/\log n$ for $i=1,\ldots, n$ and that we have $\sum_{i=1}^n d_i < 2n$. We use this to bound  $\sum_{i=1}^n d_i^2$. By standard convexity arguments, this sum is maximised if $2\log n/\maxd$ of the $d_i$ take the largest possible value $\maxd n/\log n$, so that
		\begin{eqnarray}\label{sqsu}
		\sum_{i=1}^n d_i^2 < \frac{2\log n}{\maxd} \left(\frac{\maxd n}{\log n}\right)^2 = \frac{2\maxd n^2}{\log n}.
		\end{eqnarray}
		The three following claims are concerned with bounding the expectation and showing concentration of $w_f,m_a,w_a$.
		\begin{Claim}\label{colw}
			For every colour $f\in c(K_n)$ we have $\Pr[w_f \geq \aon n]=o\left(n^{-2}\right)$.
		\end{Claim}
		\begin{proof}
			Fix some $f\in c(K_n)$. Write $w_f = X_1 +\ldots + X_{n/4}$, where $X_i$ is a random variable that takes the value $d_i$ if $f$ is present at $b_i$ and zero otherwise. We start by bounding the expectation of $w_f$. Observe that $f$ is present at at most $2\glo n$ vertices of $K_n$ since $c$ is globally $\glo n$ bounded. Recall, that when chosing $b_i$ in the $i$-th step of RGE, there are at least $n/2$ feasible choices. Thus, the probability that we choose a vertex at which $f$ is present is at most $2\glo n/(n/2) = 4\glo$. Thus, $\Pr[X_i = d_i] \leq 4\glo$. It is crucial that this bound on $\Pr[X_i = d_i]$ holds independently of the previous `history'  of RGE. By linearity of expectation, we obtain $\E[w_f] \leq 4\glo \sum_{i=1}^{n/4} d_i \leq 8\glo n \leq \aon n/2$, where we used \eqref{eq1} in the last inequality. Since the bound on $\Pr[X_i = d_i]$ holds independently of previous choices of $b_j$ and since changing the outcome of $X_i$ changes the random variable $w_f$ by at most $d_i$, we can apply Azuma-Hoeffding. Recalling \eqref{eq1} and \eqref{sqsu}, we find
			$$
			\Pr\bigl[w_f \geq \aon n\bigr]\leq \Pr\Bigl[w_f - E[w_f] \geq \aon n/2\Bigr]\leq \exp\left(-\frac{(\aon n/2)^2}{2\sum_{i=1}^{n/4} d_i^2}\right) \leq \exp\left(-\frac{\aon^2}{16 \maxd}\log n\right) = o\left(n^{-2}\right).
			$$
			This finishes the proof of the first claim.
		\end{proof}
		
		\begin{Claim}\label{verm}
			For each vertex $g\in K_n$ we have $\Pr[m_g \geq \atw n]=o\left(n^{-1}\right)$.
		\end{Claim}
		\begin{proof}
			Fix some vertex $g\in K_n$. We start by observing that not `many' vertices of $K_n$ can share `a lot of' colours with $g$. More precisely, for a vertex $b\in K_n$ let $ov(b) = |\{f\in K_n | f \text{ present at } g \text{ and } b\}|$. Since there are less than $n$ colours present at $g$ and since every colour is present at at most $2\glo n$ vertices, we can bound
			$$
			\sum_{b\in K_n} ov(b) \leq n\cdot 2\glo n = 2\glo n^2.
			$$
			Call a vertex $b\in K_n$ `bad' if $ov(b)>\bad n$. Observe that by the bound given above, we can have at most $\nbad n$ bad vertices since $\nbad n \cdot\bad n \geq 2\glo n^2$ by \eqref{eq2}.
			
			We now define random variables $X_1,\ldots, X_{n/4}$. Set 
			$$X_i = |\{j>i \;|\quad b_i b_j \in T_0\text{ and } c(b_i b_j) \text{ present at } g \}|.$$ 
			The condition $j>i$ avoids double counting edges, so that we have $m_g \leq X_1 + \ldots + X_{n/4}$. \Alexey{Is there really equality here? Doesn't $X_i$ potentially count edges from $g$ to $B$ while $m_g$ only counts edges from $g$ to $A$}\Fred{changed to inequality} To bound each $X_i$ we distinguish whether or not $b_i$ is bad. To this end, introduce random variables $Y_1,\ldots, Y_{n/4}$ where $Y_i = d_i$ if $b_i$ is bad and $Y_i = 0$ else. Moreover, let $Z_1,\ldots,Z_{n/4}$ be random variables such that $Z_i = 0$ if $b_i$ is bad and $Z_i = X_i$ otherwise. It is easy to see that $X_i \leq Y_i + Z_i$. Writing $Y = Y_1+\ldots+Y_{n/4}$ and $Z = Z_1+\ldots+Z_{n/4}$ for convenience, we thus have $m_g\leq Y+Z$. We first bound $Y$. Observe that, since there are at most $\nbad n$ bad vertices and at least $n/2$ feasible choices for RGE at the $k$-th step, we have $\Pr[Y_k = d_k]\leq 2\nbad$, equivalently $E[Y_k]\leq 2\nbad d_k$. We can now proceed analogously to the proof of Claim \ref{colw}: By linearity of expectation $E[Y] \leq 4\nbad n  \leq \atw n/4$ (using \eqref{eq2}) and thus, using Azuma-Hoeffding as well as \eqref{eq1} and \eqref{sqsu} again, 
			\begin{eqnarray}\label{one}
			\Pr[Y\geq \atw n/2] \leq 
			\Pr\Bigl[Y - E[Y] \geq \atw n/4\Bigr]\leq \exp\left(-\frac{\atw^2}{64 \maxd}\log n\right) = o\left(n^{-1}\right).
			\end{eqnarray}
			To bound $Z$ consider each summand $Z_k$ individually: Fix some $k$ and assume that $b_k$ is not bad, otherwise $Z_k = 0$. Denote $I = \{i>k|b_kb_i\in T_0\}$. For any $i\in I$, let $W_i$ be the indicator variable of the event that $c(b_k b_i)$ is present at $g$. We then have $Z_k = \sum_{i\in I} W_i$. We now bound the probability that $W_i$ is 1. Note that once $b_k$ is chosen, the event $W_i=1$ depends on the choice of $b_i$ only. Since $b_k$ is not bad, there are at most $\bad n$ colours present at $b_k$ which are also present at $g$. Moreover, there are at least $n/2$ feasible choices for $b_i$. Thus, $\Pr[W_i = 1] \leq \bad n / (n/2) = 2\bad$. It is now easy to see that, since $T_0$ is a forest and has at most $n/4$ edges, we can bound $Z_1+\ldots+Z_{n/4}$ by the sum of $n/4$ independent Bernoulli variables with success probability $2\bad$. By this coupling argument and the classic Chernoff bound (see e.g. \cite{Che}) we can use $\atw\geq 2\bad$ from \eqref{eq2} to obtain:
			\begin{eqnarray}\label{two}
			\Pr[Z \geq \atw n/2] =\Pr\left[Z\geq 2E[Z]\right] \leq \exp\left( -2E[Z]^2/n \right) =o\left(n^{-1}\right).
			\end{eqnarray}
			Combining \eqref{one} and \eqref{two}, we finally obtain the desired result:
			$$
			\Pr[m_g \geq \atw n]\leq \Pr[Y+Z\geq \atw n] \leq \Pr[Y\geq \atw n /2] + \Pr[Z\geq \atw n /2]=o\left(n^{-1}\right).
			$$
		\end{proof}
		
		\begin{Claim}\label{verw}
			For every vertex $g\in K_n$ we have $\Pr[w_g\geq \ath n]=o\left(n^{-1}\right)$.
		\end{Claim}
		\begin{proof}
			Fix some vertex $g\in K_n$. For the remainder of the proof we will assume that $m_g\leq \atw n$ and that for each colour $f$ which is present at $g$ we have $w_f \leq \aon n$. This is fine, since by Claims \ref{colw} and \ref{verm} and a union bound these inequalities hold with probability $1-o\left(n^{-1}\right)$. \\
			For this proof it will be convenient to have the  following piece of notation: Recall that for each $i\leq n/4$, there is at most one index $j<i$ so that $b_jb_i\in T_0$. If such an index $j$ exists, define $pre(i) = j$. Otherwise let $pre(i) = \emptyset$, say.\\
			Define random variables $X_1,\ldots,X_{n/4}$ as follows: If at the $k$-th step of RGE $b_k$ is chosen so that $c(gb_k)\in c(T[b_1\ldots,b_{k-1}])$, then we set $X_k = d_k$, otherwise we set $X_k = 0$. \\
			We also define random variables $Y_1,\ldots,Y_{n/4}$: If there is a colour $f = c(b_{pre(k)}b_k)$ chosen in the $k$-th step\footnote{This of course only makes sense if $pre(k)\neq \emptyset$. In case $pre(k)=\emptyset$ then no colour is chosen at the $k$-th step and we set $Y_k = 0$.  \Alexey{This footnote cannot stay as it is since it would annoy the reviewer a lot. We should either remove it completely, or preferably think of a way of saying what it says without insisting that the reader to do extra work.}\Fred{I deleted the second sentence of this footnote which said `There will be similar case distinctions in what follows and we trust the reader to resolve these problems.' It was a lie anyways.}} and some index $j<k$ so that $c(gb_j) = f$, then we set $Y_k=d_j$. Otherwise we set $Y_k = 0$. \Fred{Sentence added here:} We will slightly abuse notation and denote the event that ``$c(gb_j) = c(b_k b_{pre(k)})$ for $j<k$'' by $\{Y_k = d_j\}$, even though this is ambiguous if there is some $i\neq j$ with $d_i = d_j$.\\
			 We briefly explain why, with these definitions, we have $w_g \leq X + Y$, where $X=X_1+\ldots + X_{n/4}$ and $Y = Y_1+\ldots + Y_{n/4}$. Recall that, by definition, $w_g = \sum_{c(gb_i)\in c(T_0) } d_i$ unless $g\in\{b_1,\ldots,b_{n/4}\}$ in which case $w_g = 0$. For some fixed $k$, $d_k$ is a summand of $w_g$ if the colour $c(gb_k)$ lies in $c(T_0)$ - this can happen in one of two ways: If $c(gb_k)\in T[b_1,\ldots,b_{k-1}]$, then $X_k = d_k$, otherwise the colour $c(gb_k)$ is chosen at the $i$-th step of RGE for some $i>k$ and then $Y_i = d_k$. This shows that $d_k$ is a summand of $X+Y$ if it is a summand of $w_g$ implying $w_g \leq X + Y$. We remark that, crucially, the colour $c(gb_k)$ cannot be chosen in the $k$-th step of RGE itself, unless $g = b_{pre(k)}$ in which case $w_g = 0$ by definition. This is the only step that fails in our proof of Theorem \ref{TRT} if we replace the condition that our colouring is proper (i.e. locally 1-bounded) by the condition that it is locally 3-bounded, say.\\ 
			We will now derive concentration results for $X$ and $Y$ implying the desired result. Consider $X$ first, we use the same line of reasoning as in Claims \ref{colw} and \ref{verm}. Recall that we conditioned on the event $m_g\leq \atw n$. Thus, at the $k$-th step of RGE there are at most $m_g\leq\atw n$ vertices $b\in K_n$ so that $c(gb)\in c(T[b_1,\ldots,b_{k-1}])$. Hence, $\Pr[X_k = d_k] \leq \atw n /(n/2)=2\atw$. We deduce $E[X] \leq 4\atw n \leq \ath n/4$ (using \eqref{eq1}) and, by Azuma-Hoeffding and equations \eqref{eq1} and \eqref{sqsu},
			\begin{eqnarray}\label{x}
			\Pr[X\geq \ath n/2] \leq \Pr\Bigl[X-E[X]\geq \ath n/ 4\Bigr] =\exp\left(-\frac{\ath^2}{64\maxd} \log n\right) = o(n^{-1}).
			\end{eqnarray}
						
			It remains to bound $Y$. We start with the observation that, if $Y_k = d_i$, then there is no index $k'\neq k$ with $Y_{k'}=d_i$, since otherwise we would have $c(b_{pre(k')}b_{k'}) = c(b_{pre(k)}b_k) = c(gb_i)$ contradicting the fact that $T_0$ is rainbow. 
			
			For a more detailed analysis of $Y$, for each $1\leq k \leq n/4$ we record which values $Y_k$ can take given the choices of $b_1,\ldots,b_{k-1}$. We record these values in a list $L_k$. We add $i$ to $L_k$ if $i<k$ and the colour $c(gb_i)$ is present at $b_{pre(k)}$. Observe that, if $c(gb_i)$ is not present at $b_{pre(k)}$, then it is impossible that $c(b_{pre(k)}b_k) = c(gb_i)$ and therefore it is impossible that $Y_k = d_i$. Hence, $Y_k$ can only take the value $d_i$ if $i\in L_k$. 
			
			\Alexey{I don't think this paragraph is quite correct if $d_i=d_j$ holds for $i\neq j$. When you say things like ``$Y_k=d_i$'' you really mean the event that ``there is a colour $f = c(b_{pre(k)}b_k)$ chosen in the $k$-th step and some index $i<k$ so that $c(gb_i) = f$''? We should think of some way to not have this issue.}
			\Fred{I added a sentence above. I had the impression that being a bit sloppy makes it easier to understand. But if you think we should be formal, I can think of another way to clarify this. }
			
			If it is the case that $i\in L_k$, then we have $\Pr[Y_k = d_i]\leq 2/n$, since in the $k$-th step RGE has at least $n/2$ feasible choices for $b_k$ and at most one of them leads to $Y_k = d_i$. 			
			Moreover, since there are $n/2$ feasible choices and at most $k-1 \leq n/4$ values that $Y_k$ can take, we know that for any given $i\in L_k$ the probability that $Y_k=d_i$ is at most $4/n$ even if we condition on the event $\left\{Y_k\neq d_j \text{ for all } j\in L_k\setminus\{i\}\right\}$. Thus, the bound $\Pr[Y_k = d_i] \leq 4/n$ holds, independently of whether or not we know that $Y_k\neq d_j$ for some values of $j$ (note that this is not necessarily true for the bound $\Pr[Y_k = d_i]\leq 2/n$: If we know that $Y_k\neq d_j$, this means that RGE does not choose the edge of colour $c(gb_j)$ in the $k$-th step, so that it becomes more likely that RGE chooses the edge of colour $c(gb_i)$). 
			Hence, using a coupling argument, we can bound $Y_k \leq \sum_{i\in L_k} W_{k,i}$ where the $W_{k,i}$ are mutually independent random variables with $\Pr[W_{k,i} = d_i] = 4/n$ and $\Pr[W_{k,i} = 0] = 1-4/n$. \Alexey{I think this coupling argument needs more detail. In particular it wasn't obvious to me what is the purpose of conditioning on the events $\left\{Y_k\neq d_j \text{ for all } j\in L_k\setminus\{i\}\right\}$.}\Fred{I think without this conditioning, we can't bound $Y_k$ by a sum of independent random variables, since there are dependencies. I tried to add some explanations.}
			
			We now slightly adapt our choice of $W_{k,i}$ to reflect the fact that for each $i$ there is at most one $k$ with $Y_k = d_i$. This will allow us to get a better bound on $Y$. For each $i\in L_k$ we introduce a random variable $W_{k,i}'$. These random variables are sampled in increasing order of $k$. If $W_{k',i}'=0$ for all $k'<k$ (with $i\in L_{k'}$, of course), then we set $W_{k,i}'=W_{k,i}$. If, on the other hand, there is some $k'<k$ so that $W_{k',i}' = d_i$, then we set $W_{k,i}'=0$. Recalling that for each $i$ there is at most one $k$ with $Y_k = d_i$, we can thus use a coupling argument to bound
			$$
			Y \leq \sum_{k = 1}^{n/4} \sum_{i \in L_k} W_{k,i}'.
			$$
			Changing the order of summation gives			 
			$$
			Y \leq \sum_{i = 1}^{n/4} \sum_{\substack{k \; :\\ L_k\ni i} } W_{k,i}'.
			$$	
				
			Now, write $W_{i}' = \sum_{k : L_k\ni i} W_{k,i}'$ and $W'=\sum_{i=1}^{n/4} W_i'$. By construction, the $W_{i}'$ are mutually independent and each $W_{i}'$ either takes the value $d_i$ or 0. To conclude, we will bound $\Pr[W_i'=d_i]$ and then apply Azuma-Hoeffding to bound $W'$. To achieve the former, note that $\Pr[W_i' = d_i]= \Pr[\exists k:W_{k,i}=d_i]\leq\frac{4}{n} \cdot \left\vert \{k: L_k\ni i\}\right\vert$, where the last inequality is a union bound. It follows from the definitions of $L_k$ and $w_{c(gb_i)}$ that for every $i$ 
			$$
			\left\vert \{k:L_k\ni i\}\right\vert=
			\sum_{\substack{j \;:\\ c(gb_i) \text{ present at }b_j }}\left\vert\{k : pre(k)=j\}\right\vert\leq
		  \sum_{\substack{j \\ c(g{b_i})\text{ present at } b_j}} d_j=w_{c(gb_i)}.
			$$
			\Alexey{I added the second and third steps to the previous inequality, check that it's still correct and makes sense}\Fred{I think it is correct now} Since we conditioned on the event that $w_{c(gb_i)}\leq \aon n$, we obtain $\Pr[W_i' = d_i]\leq 4\aon$. 
			
			This implies (together with \eqref{eq1}) that $E[W']\leq 8\aon n \leq \ath n /4$. Applying Azuma-Hoeffding as well as \eqref{eq1} and \eqref{sqsu} one last time, we get
			$$
			\Pr\left[W' \geq \ath n/2\right] \leq \Pr\left[W'-E[W']\geq \ath n/4 \right] \leq \exp\left(-\frac{\ath^2}{64\maxd} \log n\right) = o(n^{-1}).
			$$
			Combining this with $Y \leq W'$, $w_g\leq X+Y$ and \eqref{x} finishes the proof.	
		\end{proof}
		Lemma \ref{RGEL} now follows from Claims \ref{colw}-\ref{verw} and a union bound over all $n$ vertices and all at most ${n\choose 2}$ colours.	
		
	\end{proof}

	\subsection{Proof of Lemma \ref{LLLS}}\label{LLL}	
	
	\begin{proof}
		We observe $|A|\geq (1-\roots) n$. By `$n$ sufficiently large' we simply mean that $\Pi_{i=0}^3 \left(|A|-i\right)\geq \left((1-\roots)n\right)^4 /2$.
		
		To prove the Lemma we apply the Local Lemma as stated in Theorem \ref{LocalS}. To this end, we consider a random bijection $\sigma:W\rightarrow A$ chosen uniformly from all such bijections. We want to find $\sigma$ so that $\pi\cup \sigma$ gives a rainbow embedding of $T$. The only obstruction to the embedding being rainbow is if two edges have the same colour. We distinguish cases depending on where these edges lie - each edge either lies in $T_0$ or goes from $B$ to $A$ or lies in $A$. We therefore have the following kinds of `bad events'.  
		
		\begin{itemize}
			\item (One edge in $T_0$, one going from $B$ to $A$) 
			For  a vertex $v\in W$ and $v_0\in R$ with $vv_0\in T$ and a vertex $a\in A$ with  $c(\pi(v_0)a)\in c(T_0)$, it is a bad event if $\sigma(v) = a$, since then $\pi\cup\sigma$ will not be a rainbow embedding of $T$. We refer to this event as $\textit{FR}_{v, v_0}^{a}$. Here in `$\textit{FR}$' the $F$ stands for `forbidden' indicating that the edge $\pi(v_0)a$, uses a `forbidden' colour from $c(T_0)$. The $R$ stands for `root' indicating that the edge $v_0 v$ is incident to the `roots' $R$.

			\item (One edge in $T_0$, one in $A$)			
			For distinct vertices $v,w\in W$ and distinct vertices $a,b\in A$ with $vw\in T$ and $c(ab)\in c(T_0)$ it is a bad event if $\sigma(v)=a,\sigma(w)=b$. We refer to this event as $F_{v,w}^{a,b}$. \\
			
			\item (Both edges from $B$ to $A$)
			If there are distinct vertices $v,w\in W$ and $v_0,w_0\in R$ so that $v_{0}v,w_{0}w\in T$ and if there are distinct vertices $a,b\in A$ with $c(\pi(v_0)a)=c(\pi(w_0)b)$, then it is a bad event if $\sigma(v)=a,\sigma(w)=b$. We refer to this event as $\textit{SRR}_{v,v_0,w,w_0}^{a,b}$. Here the `S' indicates that two edges which have the `same' colour are chosen and the `RR' indicates that both these edges are incident to $R$.
			
			\item (One edge from $B$ to $A$ and one in $A$) If there are distinct vertices $v,w,x\in W$ and $v_0\in R$ so that $v_0v, w x\in T$ and if there are distinct vertices, $a,b,d \in A$ with $c(\pi(v_0)a) = c(bd)$, then it is a bad event if $\sigma(v)=a,\sigma(w)=b,\sigma(x)=d$. We refer to this event as $\textit{SR}_{v,v_0,w,x}^{a,b,d}$.
			
			\item (Both edges in $A$) 
			If there are distinct vertices $v,w,x,y\in W$ so that $vw,xy\in T$ and if there are vertices $a,b,d,e\in A$ so that $c(ab)= c(de)$, then it is a bad event if $\sigma(v)=a,\sigma(w)=b,\sigma(x)=d,\sigma(y)=e$. We refer to this event as $S_{v,w,x,y}^{a,b,d,e}$.
		\end{itemize}
		\noindent\textbf{Notation.} For an element $x\in A\cup W$ we write $x\in \textit{FR}_{v,v_0}^a$ if $x\in \{v,a\}$. For the other bad events, we will use an analogous notation. For two bad events $C,D $ we write $C\sim D $ if there is an $x\in A\cup W$ with $x\in C$ and $x\in D $.
		
		\noindent\textbf{Remark.} Observe that, for example, the events $F_{v,w}^{a,b}$ and $F_{w,v}^{b,a}$ are identical. In the arguments below, we will count events $F_{v,w}^{a,b}$ that have $x\in F_{v,w}^{a,b}$ for some fixed $x\in W\cup A$. In these counting arguments we will count each event only once.
		If for example $x\in W$, we will count the events with labels $F_{x,w}^{a,b}$, but not those with labels $F_{w,x}^{a,b}$. The case $x\in A$ is treated similarly. 
		This remark of counting each event only once even if it has multiple `labels' also applies to the other kinds of bad events, that is to the events $\textit{SRR}_{v,v_0,w,w_0}^{a,b}$, $\textit{SR}_{v,v_0,w,x}^{a,b,d}$ and $S_{v,w,x,y}^{a,b,d,e}$.
		
		 \Alexey{Can this be made precise somehow? This remark asks the reader to do extra work, which is bad. The remark should definitely be rephrased not to do this.}\Fred{I gave one concrete example hopefully making it clearer. If this is not clear enough, I guess I would have to add a short explanation in each paragraph below. The only alternative I see would be ordering the vertices in W and only allowing events $F_{v,w}^{a,b}$ with $v<w$ but I think this makes things more cumbersome to read and write } 
		
		Observe that all these events are indeed canonical and that our notation $C\sim D $ for bad events matches the notation from Theorem \ref{LocalS}. Finally note that, if none of the bad events occurs, then the embedding $(\pi\cup\sigma)(T)$ is indeed rainbow. Denote the union of all these bad events by $\B$. Our aim is to apply Theorem \ref{LocalS} in order to show that the embedding $(\pi\cup \sigma)(T)$ is rainbow with positive probability. Thus, it suffices to show that for all $C\in\B$ we have
		\begin{eqnarray}\label{this}
		\sum_{\substack{C'\in \B \\ C'\sim C}} \Pr[C']\leq \frac{1}{4}.
		\end{eqnarray}
		Proving \eqref{this} is a bit cumbersome but not hard, as long as we count carefully. We introduce some notation and outline the straightforward approach to prove \eqref{this} before giving the details. We will define functions $\textit{fr},f,\textit{srr},sr,s: \;A\cup W \rightarrow \R$. For example $\textit{fr}(x)$ is defined as
		$$
		\textit{fr}(x) = \sum_{x\in \textit{FR}_{v,v_0}^{a}} \Pr[\textit{FR}_{v,v_0}^{a}].
		$$
		The other functions are defined analogously so that $f(x)$ features bad events of type $F$ rather than of type $\textit{FR}$, so that  $\textit{srr}(x)$ features bad event of type $\textit{SRR}$ and so on. We will bound each of these functions separately, this is done in Claims \ref{first}-\ref{last}. This is where we will use our assumptions on $w_f,m_g,w_g$ and $d(w)\leq D $ for $w\in W$.\\
		Having bounded $f(x),\textit{fr}(x),\ldots$, we will use the following simple bound for each $C\in\B$ to conclude. 
		
		\begin{align}\label{sum}
		\sum_{\substack{C'\in \B \\ C'\sim C}} \Pr[C'] \leq& 
		\sum_{x\in C} \Bigl(\textit{fr}(x)+f(x)+\textit{srr}(x)+\textit{sr}(x)+s(x)\Bigr)\nonumber\\
		\leq& 
		4 \Bigl(\max_{v\in W}\bigl(\textit{fr}(v)+f(v)+\textit{srr}(v)+\textit{sr}(v)+s(v)\bigr) \\
		&+ \max_{a\in A}\bigl(\textit{fr}(a)+f(a)+\textit{srr}(a)+\textit{sr}(a)+s(a)\bigr)\Bigr).\nonumber
		\end{align}
		
		We now give the details of bounding $\textit{fr}(x),f(x),\textit{srr}(x),\textit{sr}(x),s(x)$.
		\begin{Claim}\label{first}
			We have
			\begin{eqnarray*}
				\textit{fr}(v)\leq& D\aio (1-\roots)^{-1} \quad&\text{ for all }v\in W \text{ and }\\
				\textit{fr}(a) \leq&  \aio (1-\roots)^{-1}  \quad&\text{ for all }a\in A.
			\end{eqnarray*}
		\end{Claim}
		\begin{proof}
			Note that every event $\textit{FR}_{v,v_0}^a$ occurs with probability at most $\left((1-\roots)n\right)^{-1}$. \Alexey{Currently this is only an upper bound on the probability (since $|R|\leq rn$ rather than $|R|=rn$ in the statement of the lemma.)}\Fred{added 'at most',also in other claims} \\
			First, we count the number of events $\textit{FR}_{v,v_0}^a$ where $v$ is fixed: 
			Since $d(v)\leq \degree $ there are at most $\degree$  ways to choose $v_0\in R$ so that $v_0v\in T$. 
			For a fixed choice of $v_0$ there are at most $\left\vert\{a\in A :c(\pi(v_0)a)\in c(T_0)\}\right\vert =m_{\pi(v_0)}\leq \aio n$ choices of $a\in A$ so that $c(\pi(v_0)a)\in T_0$. Thus, all in all there are at most $\degree \cdot\aio n$ events of the form $\textit{FR}_{v,v_0}^a$. Hence $\textit{fr}(v)\leq D\aio (1-\roots)^{-1}$.\\
			Next, we count the number of events  $\textit{FR}_{v,v_0}^a$ where $a$ is fixed: We first choose $v_0$ so that $c(\pi(v_0)a)\in c(T_0)$. Given the choice of $v_0$, we can choose $v$ to be any one of the neighbors of $v_0$ that lie in $W$. There are at most $d(v_0)$ such neighbors. Summing over all $v_0$ with $c(\pi(v_0)a)\in c(T_0)$ we obtain a bound of $\sum_{\substack{v_0\in R \\c(a\pi(v_0))\in c(T_0) } } d_T(v_0)=w_a\leq \aio n$ possible choices \Alexey{Where relevant in the other claims can you rewrite the definitions of $w_g, w_f$, and $m_g$ with the names of the vertices and colours used in the claims (to make it more obvious to the reader why the bounds hold)}\Fred{changed this}. Hence, $\textit{fr}(a)\leq \aio (1-\roots)^{-1}$.
		\end{proof}
		
		\begin{Claim}
			We have
			\begin{eqnarray*}
				f(v)\leq& 4\degree \glo r(1-\roots)^{-2} \quad&\text{ for all }v\in W \text{ and }\\
				f(a) \leq&  4\aio (1-\roots)^{-2}  \quad&\text{ for all }a\in A.
			\end{eqnarray*}
		\end{Claim}
		\begin{proof}
			The probability of an event $F_{v,w}^{a,b}$ is at most $\Pi_{i=0}^1 ((1-\roots)n -i)^{-1} \leq 2 ((1-\roots)n)^{-2}$, where we use that $n$ is sufficiently large. \\
			Fix some $v\in W$. We first count the number of events $F_{v',w}^{a,b}\ni v$. Without loss of generality assume $v=v'$. Since $d(v)\leq \degree $, there are at most $\degree$  choices of $w$. 
			There are at most $\glo n\cdot \roots n $ edges in $K_n$ whose colour lies in $c(T_0)$, so that there are at most $2\glo r n^2$ choices for the ordered pair $a,b$ so that $c(ab)\in c(T_0)$. All in all, there are at most $\degree \cdot 2\glo r n^2$ events of the form $F_{v,w}^{a,b}$ implying $f(v)\leq 4\degree \glo r(1-\roots)^{-2}$.\\
			Now fix $a\in A$. We count the number of events $F_{v,w}^{a',b}\ni a$. Without loss of generality assume $a=a'$. There are at most $\left\vert\{b\in A :c(ab)\in c(T_0)\}\right\vert = m_a \leq \aio n$ ways to choose $b$ and, since there are at most $n$ edges in $T$, there are at most $2n$ ways to choose the (ordered) pair $v, w$ implying $f(a)\leq 4\aio (1-\roots)^{-2}$.\\
		\end{proof}
		
		\begin{Claim}\label{familiar}
			We have
			\begin{eqnarray*}
				\textit{srr}(v)\leq& 2\degree\aio (1-\roots)^{-1} \quad&\text{ for all }v\in W \text{ and }\\
				\textit{srr}(a) \leq&  2\aio (1-\roots)^{-2} \quad&\text{ for all }a\in A.
			\end{eqnarray*}
		\end{Claim}
		\begin{proof}
			The probability of an event $\textit{SRR}_{v,v_0,w,w_0}^{a,b}$ is at most $\Pi_{i=0}^1 ((1-\roots)n -i)^{-1} \leq 2 ((1-\roots)n)^{-2}$.\\
			Fix $v\in W$. We count the number of events $\textit{SRR}_{v',v_0,w,w_0}^{a,b}\ni v$. Without loss of generality assume $v=v'$. Since $d(v)\leq \degree $, there are at most $\degree$  choices of $v_0$. Now, there are at most $(1-\roots)n$ choices for $a$. Given the choices of $v_0$ and $a$, write $e=\pi(v_0)a$. We know that the edge $\pi(w_0)b$ has to have colour $c(e)$. We must choose $w_0$ so that $c(e)$ is present at $w_0$. Given such a $w_0$ there is at most one choice of $b$ so that $c(\pi(w_0)b)=c(e)$ and there are at most $d(w_0)$ choices for $w$. Summing over the possible choices of $w_0$ we see that there are at most  $\sum_{\substack{w_0\in R \\ c(e)\text{ present at } \pi(w_0)}} d\left(w_0\right) = w_{c(e)}\leq \aio n$ possible choices for the triple $w_0,w,b$. All in all, we can bound the number of events $\textit{SRR}_{v,v_0,w,w_0}^{a,b}\ni v$ by $\degree (1-\roots)n \cdot \aio n$ implying $f(v)\leq 2\degree\aio (1-\roots)^{-1}$.\\
			Now fix $a\in A$ and count the number of events $\textit{SRR}_{v,v_0,w,w_0}^{a',b}\ni x$. Without loss of generality assume $a=a'$. There are at most $n$ ways to choose a pair $v, v_0$ since each such pair corresponds to an edge of $T$. Then, just as before there are at most $\aio n$ ways to choose the triple $w_0,w,b$. Thus, $\textit{srr}(a)\leq 2\aio (1-\roots)^{-2}$.\\
		\end{proof}
		
		\begin{Claim}
			We have
			\begin{eqnarray*}
				\textit{sr}(v)\leq& 8\degree  \glo (1-\roots)^{-2} \quad&\text{ for all }v\in W \text{ and }\\
				\textit{sr}(a) \leq&   4\glo(1-\roots)^{-3} + 4\aio (1-\roots)^{-2}  \quad&\text{ for all }a\in A.
			\end{eqnarray*}
		\end{Claim}
		\begin{proof}
			The probability of an event $\textit{SR}_{v,v_0,w,x}^{a,b,d}$ is at most $\Pi_{i=0}^2 ((1-\roots)n -i)^{-1} \leq 2 ((1-\roots)n)^{-3}$.\\
			As usual we count the number of $\textit{SR}_{v',v_0,w,x}^{a,b,d}\ni v$ for a fixed $v\in W$. Here we need to distinguish two cases, either $v = v'$ or $v\in\{w,x\}$. Consider the case $v = v'$ first. As before, there are at most $\degree$  ways to choose $v_0$, then there are at most $(1-\roots)n$ ways to choose $a$. We now know that $b,d$ have to be chosen in a way so that $c(bd) = c(\pi(v_0)a)$. Since there are at most $\glo n$ edges of colour $c(\pi(v_0)a)$ this can be done in at most $2\glo n$ ways, here we count the number of ordered pairs $b,d$. Finally, there are at most $n$ ways to choose the unordered pair $w,x$. All in all, we have at most $\degree \cdot (1-\roots)n  \cdot 2\glo n \cdot n = 2\degree\glo (1-\roots) n^3$ choices. We now consider the case where $v\in \{w,x\}$. Assume $v=w$ without loss of generality. We have at most $\degree$ choices for $x$ and at most $n$ choices for the pair $v_0,v'$ and another at most $(1-\roots)n$ choices for $a$. Just as before, we are left with at most $2\glo n$ choices for the ordered pair $b,d$ so that $c(bd) = c(\pi(v_0)a)$. In total, we have $\degree \cdot n\cdot (1-\roots)n  \cdot 2\glo n = 2\degree\glo (1-\roots)n^3$ choices. Combining the two cases $v = v'$ and $v=w$, we obtain $\textit{sr}(v_r)\leq 8\degree  \glo (1-\roots)^{-2}$. \\
			We now count $\textit{SR}_{v,v_0,w,x}^{a',b,d}\ni a$ for a fixed $a\in A$. We first consider the case where $a=a'$ and then the case where $a\in \{b,d\}$. So let $a=a'$. There are at most $n$ choices for the pair $v,v_0$ and by the same argument as before, there are at most $n\cdot 2\glo n$ ways to choose $w,x$ and $b,d$. In total, we have at most $2\glo n^3$ choices. 
			We now consider the case $a\in\{b,d\}$. Assume $a=b$ without loss of generality. There at most $2n$ ways to choose ordered $w,x$ and at most $(1-\roots)n$ ways to choose  $d\neq b$. Once these choices have been made, we have to pick $v_0,v,a'$ so that $c(\pi(v_0)a')=c(bd)$. By a familiar argument (see proof of Claim \ref{familiar}), this can be done in at most $w_{c(bd)}\leq \aio n$ ways. In total we have at most $2n\cdot (1-\roots)n \cdot\aio n$ possible choices. Combining the cases $a=a'$ and $a=b$ we obtain $\textit{sr}(x)\leq 4\glo(1-\roots)^{-3} + 4\aio (1-\roots)^{-2}$.
		\end{proof}
		
		\begin{Claim}\label{last}
			We have
			\begin{eqnarray*}
				s(v)\leq& 4\degree \glo(1-\roots)^{-2} \quad&\text{ for all }v\in W \text{ and }\\
				s(a) \leq&  8\glo (1-\roots)^{-3}  \quad&\text{ for all }a\in A.
			\end{eqnarray*}
		\end{Claim}
		\begin{proof}
			The probability of an event $S_{v,w,x,y}^{a,b,d,e}$ is at most $\Pi_{i=0}^3 ((1-\roots)n -i)^{-1} \leq 2 ((1-\roots)n)^{-4}$. \\
			We start by counting the number of events $S_{v',w,x,y}^{a,b,d,e}\ni v$ for a fixed $v\in W$. Without loss of generality assume $v=v'$. There are at most $\degree$  choices for $w$ and at most $n$ choices for the unordered pair $x,y$. There are at most $\left((1-\roots)n\right)^{2}$ choices for the ordered pair $a,b$. Given the choice of $a,b$ we have at most $2\glo n$ choices for the ordered pair $d,e$. In total we have at most $\degree \cdot n\cdot \left((1-\roots)n\right)^2 \cdot 2\glo n$ so that $s(v)\leq 4\degree \glo(1-\roots)^{-2}$.\\
			We now count the number of events $S_{v,w,x,y}^{a',b,d,e}\ni a$ for a fixed $a\in A$. Without loss of generality assume $a=a'$. We have at most $4 n^2$ possibilities to choose the ordered pairs $v,w$ and $x,y$, at most $(1-\roots)n $ choices for $b$ and at most $\glo n$ choices for the unordered pair $d,e$ giving a total of $4 n^2\cdot (1-\roots)n \cdot \glo n$ choices so that $s(a)\leq 8\glo (1-\roots)^{-3}$.
		\end{proof}
		We now combine the Claims \ref{first}-\ref{last} with \eqref{sum}. To simplify, we estimate $(1-\roots)^{-1}\leq (1-\roots)^{-2}\leq (1-\roots)^{-3}$ as well as $\roots\leq 1$. Altogether, from \eqref{sum2} we obtain
		\begin{equation*}
		\sum_{\substack{C'\in \B \\ C'\sim C}} \Pr[C'] \leq 4(1-r)^{-3}	\bigl(\aio(3D + 11) + \glo(16D + 12)\bigr)\leq 1/4
		\end{equation*}
		finishing the proof of the Lemma.	
	\end{proof}
	\Alexey{This remark probably be either removed or expanded and moved to the conclusion.} 
	\Fred{Removed remark saying:\\ If need be, one can state a slightly more precise version of this Lemma. Rather than defining the weight of a colour $w_f$, one can define the weight of an edge $e\in E(K_n)$ as follows:
	\begin{eqnarray*}
		w_e = \sum_{\substack{b\in B\setminus e \\ c(e)\text{ present at } b}} d_T\left(\pi^{-1}(b)\right)
	\end{eqnarray*}
	Rather than bounding the weight of each edge by $w_e\leq \aio n$, it suffices to bound $\sum_{e \text{ incident to } a} w_e\leq \aio n^2$ for each $a\in V(K_n)$.}	
	
	\subsection{Proof of Theorem \ref{TRT}}\label{TRP}
	We now combine the previous results to prove Theorem \ref{TRT}. 
	
	Recall that $d_1,\ldots,d_{n/4}$ are the $n/4$ highest degrees of $T$, and $d_{n/4+1},\ldots,d_n$ are the $3n/4$ lowest degrees of $T$. We start by observing that for every $i>n/4$ we have $d_i\leq 4$. To see this, observe $\sum_{i=1}^n d_i<2n$, moreover $d_i\geq 1$ for all $i$ so that $\sum_{i=1}^{n/4}d_i<5n/4 $. Thus the average degree of the vertices $v_1,\ldots,v_{n/4}$ is less than $5$, which implies that $d_i\leq 4 $ for all $i>n/4$.
	
	Next, we choose $\glo = \maxd = 2^{-38}, \aon =2^{-16}, \atw =2^{-15}, \ath =2^{-11}, \bad =2^{-16}, \nbad =2^{-19}$ and observe that the conditions of Lemma \ref{RGEL} are satisfied. 
	
	We now apply Lemma \ref{RGEL} to obtain a partial embedding of the $n/4$ vertices $v_1,\ldots,v_{n/4}$ of $T$ with highest degrees. Now take $\aio = \ath$. It then follows from Lemma \ref{RGEL} that the partial embedding satisfies $m_a\leq \aio n, w_a\leq \aio n$ for each $a\in V(K_n)$ and $w_f\leq \aio n$ for every $f\in c(K_n)$. Set $r = 1/4$ and $D=4$. It is easily checked that all the conditions of Lemma \ref{LLLS} are satisfied so that we can find a rainbow embedding of $T$ into $K_n$. This finishes the proof of Theorem \ref{TRT}.

	\section{Concluding Remarks}\label{conc}
	In this section we make some remarks and present open questions, some of which are directly related to our results and some of which are inspired by them and might lead to further research. 
	
	Let us begin by describing some constructions of colourings of $K_n$ and $n$-vertex trees $T$, so that $K_n$ does not contain a rainbow copy of $T$.	
	These constructions show that in order to find rainbow embeddings of $n$-vertex trees into $K_n$, we need stronger conditions than the colouring to just be proper. Specifically, we cannot simply drop the condition of Theorem \ref{TRT}, that the colouring be globally $\alpha n$-bounded for some $\alpha<1/2$.
	
	We start with the observation that a 1-factorisation of $K_{2n}$ contains precisely $2n-1$ colours, which equals the number of edges of a spanning tree. Thus, any rainbow spanning tree within a 1-factorisation has to use each colour precisely once.
	
	The first construction comes from Maamoun and Meyniel \cite{m&m}, who give a colouring of $K_n$  for $n=2^k$ which does not contain a rainbow Hamilton path. We adapt their argument to prove the following result.
	\begin{prop}\label{Prop_Construction1}
		For $n=2^k$, there is a 1-factorisation of $K_{n}$ which does not contain a rainbow spanning tree in which all but precisely two degrees are odd. The same 1-factorisation does not contain a rainbow Hamilton path.
	\end{prop}
	\begin{proof}
		Suppose $n=2^k$ and label the vertices of $K_n$ by elements of the finite group $\Z_2^k$. For two distinct vertices $a,b\in V(K_n)$ colour the edge $ab$ by the colour $a+b$, so that the set of colours used is precisely $\Z_2^k\setminus\{0\}$. It is not hard to see that this colouring is a 1-factorisation. Observe also that the sum over all colours is $0$. Thus, when we sum over the colours of the edges of a rainbow spanning rainbow tree we obtain 0. Now suppose $T$ is a spanning tree, embedded into $K_n$. Summing over all colours of the edges of $T$ and reordering, we obtain
		\begin{eqnarray*}
			\sum_{ab \in E(T)} c(ab) = \sum_{ab \in E(T)} {(a+b)} = \sum_{a \in V(T)} d_T(a)a.
		\end{eqnarray*}
		If $T$ is a tree in which all but precisely two degrees are odd, say $x,y\in V(K_n)$ are the vertices with even degree, then 
		\begin{eqnarray*}
			\sum_{ab \in E(T)} c(ab) = \left(\sum_{a \in V(T)} a\right) - x -y = x+y. 
		\end{eqnarray*}
		This expression never equals zero, no matter how we choose $x$ and $y$. This contradicts the observation made above, that the sum should be 0. The same argument works if we replace the condition of all but two vertices having odd degree by the condition of all but two vertices having even degree. The only spanning tree fulfilling the latter is a Hamilton path and in this case, our argument is that of Maamoun and Meyniel \cite{m&m}.
	\end{proof}
	We also show that the colouring from the proof above is not the only colouring which does not allow for a rainbow embedding of certain trees.
	\begin{prop}\label{Prop_Construction2}
		For every $n=2k$, there is a tree on $n$ vertices which has no rainbow embedding into any 1-factorised $K_n$.
	\end{prop}
	\begin{proof}
		Let $S$ be a tree on $2k$ vertices which has two distinct vertices $x,y$ with the following properties. Firstly, $x$ is not adjacent to $y$ in $S$. Secondly every edge in $S$ is incident to $x$ or $y$. 
		
		There are many trees $S$ satisfying these two properties. Any such tree consists of two stars rooted and $x$ and $y$ overlapping in one vertex.
		
		Let $c$ be a 1-factorisation of $K_n$ and suppose we are given an embedding of $S$ into $K_n$. We abuse notation to denote by $x,y$ not only the vertices of $S$, but also their images in $K_n$. In our embedding, no edge has colour $c(xy)$, since the edge $xy$ is absent in $S$ and since every edge of $S$ is adjacent to $x$ or $y$ and therefore has a colour different from $xy$. Thus, the embedding does not use each colour of $K_n$ and cannot be rainbow.	
	\end{proof}

The above discussion shows that the ``global $\glo n$-boundedness'' condition in  Theorem \ref{TRT} cannot be completely removed. In cotrast to this we expect that the assumption on the maximum degree of $T$ in Theorem \ref{TRT} is unnecessary---we believe that there is a constant $\glo>0$ so that any proper, globally $\glo n$-bounded colouring of $K_n$ contains a rainbow copy of \emph{every} $n$-vertex tree $T$.  As evidence for this, we can prove that there is a small constant $\glo$ so that any globally $\glo n/\log^3 n$-bounded colouring contains a rainbow copy of every $n$-vertex tree (see~\cite{MastersThesis}).
	
	Another direction one can take is to see how close we can come to embedding spanning trees into a properly coloured $K_n$ without further conditions on the colouring. More precisely, one can ask whether there is a function $f(n)=o(n)$ so that any tree on $n-f(n)$ vertices has a rainbow embedding into any properly coloured $K_n$. This would for example generalise the result by Alon, Pokrovskiy and Sudakov \cite{longcycle} about almost spanning paths in properly coloured $K_n$.
	
	Finally, Theorem \ref{TLRC} invites to thinking about whether proper colourings of certain tournaments (e.g. regular tournaments) contain long rainbow cycles/paths. There is a great amount of conjectures and results in the area of Hamilton cycles in regular tournaments, see e.g. \cite{kuhn} and more recent results by the same authors, but so far little effort has gone into investigating rainbow structures in this setting. 
	
	\subsubsection*{Acknowledgment.}
	The constructions in Propositions~\ref{Prop_Construction1} and~\ref{Prop_Construction2} were also found independently by Richard Montgomery.
		
	\bibliography{mybib2}
	\bibliographystyle{siam}
		
\end{document}